\documentclass[reqno]{amsart}
\usepackage{amssymb, amsmath}
\usepackage{natbib}
\usepackage{lscape}
\usepackage{dsfont}
\usepackage{mathrsfs}
\usepackage{bbm}
\usepackage[hmargin=3cm,vmargin={3.8cm,4cm}]{geometry}
\usepackage[pdftex,plainpages=false,colorlinks,hyperindex,bookmarksopen,linkcolor=red,citecolor=blue,urlcolor=blue]{hyperref}

\DeclareMathAlphabet{\mathpzc}{OT1}{pzc}{m}{it}

\bibpunct{[}{]}{;}{n}{,}{,}

\newtheorem{te}{Theorem}[section]
\newtheorem{prop}[te]{Proposition}

\numberwithin{equation}{section}

\theoremstyle{definition}
\newtheorem{defin}[te]{Definition}

\theoremstyle{remark}

\newtheorem{os}[te]{Remark}

\allowdisplaybreaks

\def\bbR{{\mathbb R}}

\def\bbV{{\mathbb V}}

\def\bbR{{\mathbb R}}

\AtBeginDocument{\let\origref\ref
   \renewcommand{\ref}[1]{(\origref{#1})}}

\begin{document}

	\title[]{On discrete-time semi-Markov processes }
	\date{\today}

	\author{Angelica Pachon${}^1$}
	\address{${}^1$Faculty of Computing, Engineering and Science, University of South Wales, UK}
	
	\author{Federico Polito${}^2$}
	\address{${}^2$Mathematics Department ``G.~Peano'', University of Torino, Italy}
	
	\author{Costantino Ricciuti${}^2$}

	\begin{abstract}

		In the last years, several authors studied a class of continuous-time  semi-Markov processes
		obtained by time-changing  Markov processes by  hitting times of independent subordinators. Such processes are governed by integro-differential convolution equations of generalized fractional type.
		The aim of this paper is to develop a discrete-time counterpart of such a theory and to show relationships and differences with respect to the continuous time case.
		We  present  a class of discrete-time semi-Markov chains which can be constructed as time-changed Markov
		chains and we obtain the related governing convolution type equations.
		Such processes converge weakly to those in continuous time  under suitable scaling limits.
	\end{abstract}

	\maketitle
		
		\tableofcontents

		\section{Introduction}

			It is well-known that the memoryless property of homogeneous Markov processes imposes restrictions
			on the waiting time spent in a state, which must be either exponentially distributed (in
			the continuous time case) or geometrically distributed (in the discrete time case). Indeed, the exponential and the geometric distributions are the only to enjoy the lack of memory property.
			However, in many applied models it is useful to  relax the Markov assumption in order to allow
			arbitrarily distributed waiting times in any state. This leads to semi-Markov
			processes. The theory of semi-Markov processes was introduced by L\'evy \cite{levy} and Smith \cite{smith}  and  was developed in many subsequent works, such as \cite{cinlarsemi,fellersemi,gihman,jacod,koro,kurtz,pykefinite,pykeinfinite}.
			  
			Since their introduction, semi-Markov processes have been mostly studied in the continuous time case, while discrete time processes are rarer in the literature (see e.g.\ \cite{barbu} and the references therein). Time is usually assumed to be continuous, even if some physical theories claim that it could be discrete; however, it is true that we observe nature at discrete time instants.
			Moreover, in many applications the time scale is intrinsically discrete. For instance, in DNA analysis, any approach  is based on discrete time because one deals with (discrete) sequences of four bases: A, T, C, G. Also, the techniques of text recognition are based on discrete-time models. In reliability theory, one could be interested in the number of times (e.g.\ days, hours) that a specific event occurs. Thus, discrete-time semi-Markov processes undoubtedly deserve a  more in-depth analysis and the aim of the present paper is to give a contribution in this direction. 

			In this paper we study the  large class of semi-Markov chains which is generated by time-changing a discrete-time Markov chain by discrete-time renewal counting processes, that is with compositional inverses of increasing random walks with i.i.d.\ positive integer jumps.			
			A little variant of such processes is proved to be governed by equations of the form
			\begin{align}
				\sum _{\tau =0}^\infty \bigl ( p(x,t)-p(x,t-\tau) \bigr) \mu (\tau) &=\mathcal{G}_x p(x,t), \qquad t\in \mathbb{N}, \label{seconda equazione generale}
			\end{align}
			where $p(x,t)$ is the probability that the process is in state $x$ at time $t$, $\mu$ is a probability mass function supported on the positive integers and $\mathcal{G}_x$ is related to the generator of the original Markov chain. In the case in which the waiting times follow a discrete Mittag--Leffler distribution (see Section \ref{remarkable} for its definition), equation \ref{seconda equazione generale} reduces to
			\begin{align}
				(I-\mathcal{B})^\alpha p(x,t)=\mathcal{G}_x p(x,t) \qquad t\in \mathbb{N}, \qquad \alpha \in (0,1), \label{intro b}
			\end{align}
			where $\mathcal{B}$ is the backward shift operator in the time-variable, that is $\mathcal{B}p(x,t)=p(x,t-1)$, while $I-\mathcal{B}$ is the discrete-time derivative and $(I-\mathcal{B})^\alpha$ is its fractional power.

		   	By this construction,  a class of discrete-time semi-Markov processes arises, which  enjoys interesting statistical properties and is suitable  for some types of applications in which time is intrinsically discrete. For instance, in \cite{barbu}, the authors construct nonparametric estimators for the main characteristics of discrete-time semi-Markov systems and study their asymptotic behavior; they also show some applications in  reliability theory and biology. Note that the mathematical techniques in \cite{barbu} are specific to discrete time and do not fit well with the continuous time case.

			Recently, the theory of semi-Markov processes in continuous time has regained much interest. The literature concerning this topic is vast: see for example \cite{becker,KoloCTRW,kolokoltsov,meerscheflimctrw,meertri,meerstra,meertoa,straka}. Consult also  \cite{Garra,o,ricciuti}, where the theory has been extended to models of motions in heterogeneous media. Following the theory, a semi-Markov process of this kind can be obtained by time-changing  a Markov process by the inverse hitting time of a subordinator and can be seen as scaling limits of continuous-time random walks (CTRWs).
			Such processes are known to be governed (see the review in Section \ref{rev}) by integro-differential equations, which,  in the most general  form, can be written as
			\begin{align}
				\int _0^\infty (p(x,t)-p(x,t-\tau))\nu (d\tau)= G_xp(x,t), \qquad t\in \mathbb{R}^+, \label{equazione generale}
			\end{align}
			where $G_x$ is the generator of the original Markov process, $\nu$ is the L\'evy measure of the underlying subordinator, while the operator on the left-hand side is usually called generalized fractional derivative.
 The reason of this name is that in the case where the random time is an inverse $\beta$-stable subordinator, equation \ref{equazione generale} reduces to
\begin{align}
\frac{\partial ^\beta}{\partial t^\beta}p(x,t)= G_xp(x,t), \qquad t\in \mathbb{R}^+, \label{zaza}
\end{align}
where $\frac{d^\beta}{dt^\beta}$ is the Caputo fractional derivative of order $\beta \in (0,1)$.

			Note that, formally, \ref{equazione generale} is the continuous version of  \ref{seconda equazione generale}, in the same way as \ref{zaza}  is the continuous version of \ref{intro b}.
			Here we prove that a similar but non equivalent scheme holds in discrete time.
			In fact our treatise on the discrete-time case goes actually beyond the classical scheme of the standard CTRW-theory, essentially for three reasons.
			First, the class of discrete-time processes studied in this paper is not trivially given by sampling continuous-time semi Markov processes at integer times (see, for example, the discussion in remark \ref{abc}).
			Second, this paper includes  the case where the original process $X(n)$ is a generic discrete-time Markov chain on a general state space, thus without restricting us the to random walks with i.i.d.\ jumps in $\mathbb{R}^d$. Third, we show that a connection with (generalized) time-fractional equations also exists in discrete time, and not only after the continuous-time limit.  

			The structure of the paper is the following. In Section \ref{nuova}, besides a brief literature overview, we introduce the time change of discrete-time Markov chains and we present renewal chains and their inverse counting processes; special attention is  devoted to the Bernoulli process  (related to identical sequential  trials) and to the Sibuya process (related to sequential trials with memory), which are discrete-time approximations of the Poisson process and the inverse stable subordinator, respectively.
			Section \ref{main} contains our main results on discrete-time semi-Markov chains and generalized fractional difference equations, together with some results on path space convergence in the continuous time limit.  Finally, Section \ref{remarkable} is devoted to the so-called fractional Bernoulli counting processes, which are  discrete-time approximations of the  fractional Poisson process studied in several papers such as \cite{Beghin2,Beghin,Kumar1,mainardi,meerpoisson}.

\section{Literature overview and preliminary results}\label{nuova}

\subsection{Time-change of discrete-time processes}

For the sake of clarity, we  recall two important definitions in probability theory: $n$--divisibility and infinite divisibility. A random variable $X$ is said to be $n$--divisible if there exist i.i.d. random variables $Y_1, Y_2, ..., Y_n$, such that 
\begin{align*}
X\overset{d}{=} Y_1+Y_2+.... +Y_n.
\end{align*}
Instead, a random variable $X$ is said to be infinitely divisible if for each $n\in \mathbb{N}$ there exist i.i.d. random variables $Y^n_1, Y^n_2, ..., Y^n_n$ such that 
\begin{align*}
X\overset{d}{=} Y^n_1+ Y^n_2+ ... + Y^n_n.
\end{align*}
Both definitions are actually connected to stochastic processes  with stationary and independent increments. We briefly recall some related facts, for a more in-depth discussion  consult, for example, [\cite{steutel}, Chapter 1].
On  one hand, $n$--divisibility is related to  random walks defined by the partial sums
\begin{align}
X(n) = \sum _{j=1}^n \Delta_j \qquad n\in \mathbb{N}, \label{primo random walk}
\end{align}
where the $\Delta_j$ are i.i.d. random variables.

On the other hand, the notion of infinite divisibility is intimately related to  L\'evy processes.
The definition of L\'evy process makes sense only in continuous time case. Indeed, a process is called L\'evy  if, besides having independent and stationary increments, it is continuous in probability, i.e. $X(s)\overset{p}{\to} X(t)$ if $s\to t$ (see, for example, \cite{apple}  for basic notions).

 L\'evy processes with non decreasing sample paths are  called subordinators (for an overview see \cite{bertoins}). They are often used as models of random time for the construction of time-changed Markov processes (on this point consult, for example, \cite{apple}, \cite{satolevy},  \cite{librobern} and \cite{toaldopota}). A subordinator $\sigma (t)$ has Laplace transform
\begin{align*}
\mathbb{E} e^{-\eta \sigma (t)}= e^{-tf(\eta )}, \qquad \textrm{with} \,\,\,\,\, f(\eta)= b\eta+\int _0^\infty (1-e^{-\eta x}) \nu (dx), \qquad \eta >0,
\end{align*}
where $b>0$ and  the L\'evy measure $\nu$ is such that
$
\int _{(0,\infty)} (x \wedge 1) \nu (dx)<\infty.
$
We let $b=0$ and only consider strictly increasing subordinators, i.e. those such that $\int _0^\infty \nu (dx) = \infty$.  The inverse hitting time process
\begin{align*}
L(t)= \inf \{ x\geq 0 : \sigma (x) > t\}= \sup  \{ x\geq 0: \sigma (x)\leq t\}
\end{align*}
is non decreasing, and thus it can be used as a random time to construct time-changed processes.

In particular,   if $\{M(t)\}_{t \in \mathbb{R}^+}$ is a Markov process, the composition $\{M(L(t))\}_{t \in \mathbb{R}^+}$ is a semi-Markov process, which has been deeply studied, having an interesting connection to many different topics, such as anomalous diffusion, continuous time random walk limits and integro-differential and fractional equations (see \cite{merbook}, chapter 8 of \cite{kolokoltsov} and the references therein).

One of the main goals of this paper is to develop some aspects of the theory of time-change for discrete time Markov processes.
The discrete-time analogue of a subordinator is a random walk of type \ref{primo random walk} with positive integer jumps
\begin{align}
\sigma _d (n) = \sum _{j=1}^n Z_j, \qquad n\in \mathbb{N}, \qquad Z_j \in \mathbb{N}, \qquad \sigma _d(0)=0. \label{subordinatore discreto}
\end{align}
Its inverse process is given by 
 \begin{align}
L_d(t)= \max \{ n\in \mathbb{N}_0  : \sigma _d (n) \leq t    \}, \qquad t \in \mathbb{N}_0, \label{inverso subordinatore discreto}
 \end{align}
 where $\mathbb{N}_0=\mathbb{N}\cup \{0\}$.
 Indeed $L_d(t)=0$ for $0\leq t<Z_1$, $L_d(t)=1$ for $Z_1\leq t<Z_1+Z_2$, and so forth.

So, given a discrete-time Markov chain $\{ \mathcal{X}(t)\}_{t \in \mathbb{N}_0}$ and  an independent process  $\{L_d(t)\}_{t \in \mathbb{N}_0}$ of type \ref{inverso subordinatore discreto}, we here study the time-changed process $\{ \mathcal{X} (L_d(t))\}_{t \in \mathbb{N}_0}$.
We refer to \cite{steutel} for some definitions and techniques on the time change of discrete-time processes.

\subsection{Brief review on continuous-time semi-Markov chains} \label{rev}

We here summarize some facts known in the literature concerning the continuous-time case. Let us consider a continuous-time Markov chain $\{X(t)\} _{t \in \mathbb{R}^+}$  on the discrete space $\mathcal{S}$
\begin{align}
X(t)=X_n, \qquad     V_n \leq t<V_{n+1}, \qquad \textrm{with} \qquad V_0=0 ,\quad V_n= \sum _{k=0}^{n-1}E_k,
\label{definizione processo Markov}
\end{align}
where $\{X_n\}_{n=0,1,\dots}$, is a discrete-time Markov chain in $\mathcal{S}$, whose stochastic matrix is defined as
\begin{align*}
H_{ij}=P(X_{n+1}=j|X_n=i)
\end{align*}
and the waiting times $E_k$ are  exponentially distributed:
\begin{align}
P(E_k>t |X_k=i)=e^{-\lambda _it} \qquad t\geq 0. \label{distribuzione esponenziale}
\end{align}
The transition probabilities 
\begin{align*}
p_{ij}(t)= P(X(t)=j|X(0)=i), \qquad i,j \in \mathcal{S},
\end{align*}
are known to solve (see for example \cite{norris}) the Kolmogorov backward equations 
\begin{align}
\frac{d}{dt} p_{ij}(t)&= \sum _{l\in S} \lambda _i(H_{il}-\delta _{il})p_{lj}(t),  \label{equazioni kolmogorov classiche}  \qquad  p_{ij}(0) =\delta_{ij},
\end{align}
as well as the related forward equations.

We now consider a semi-Markov process  $\{Y(t)\} _{t \in \mathbb{R}^+}$ which  is constructed in the same way of \ref{definizione processo Markov} except for the distribution of the waiting times, which are no longer exponentially distributed:
\begin{align}
Y(t)=X_n \qquad     T_n \leq t<T_{n+1} \qquad \textrm{where}\quad T_0=0 ,\quad T_n= \sum _{k=0}^{n-1}J_k
,
\label{definizione processo semi-Markov}
\end{align}
where the waiting times follow an arbitrary continuous distribution
\begin{align*}
P(J_k>t|X_k=i)=\overline{F}_i(t).
\end{align*}
We are interested in a particular subclass of \ref{definizione processo semi-Markov}: taking any strictly increasing subordinator $\{\sigma (t)\}_{t\in \mathbb{R}^+}$, independent of $\{X(t)\}_{t\in \mathbb{R}^+}$, we assume  that the waiting times $J_0, J_1,\dots$ are such that
\begin{align}
P(J_k>t|X_k=i)=P(\sigma (E_k)>t|X_k=i), \label{intertempo caso continuo}
\end{align}
 where $E_k$ are distributed as in  \ref{distribuzione esponenziale}.
The main result concerning such a class of semi-Markov processes (which also justifies the choice of the compound exponential  law $J_k=\sigma (E_k)$ for the waiting times) is the following: let $L(t)$ be the right continuous inverse process of $\sigma (t)$, the following time-change relation holds: 
\begin{align*}
\{Y(t)\}_{t\in \mathbb{R}^+}=\{X(L(t))\}_{t \in \mathbb{R}^+}, 
\end{align*}
where $X$ and $L$ are independent.
A heuristic proof of this fact is the following: by using the definition \ref{definizione processo Markov}, we have
 \begin{align}
X(L(t))=X_n \qquad     V_n \leq L(t)<V_{n+1}
\end{align}
namely 
\begin{align*}
X(L(t))=X_n \qquad  \sigma (V_n^-) \leq \ t<\sigma (V_{n+1})
\end{align*}
and thus the waiting times are such that $J_n = \sigma (V_{n+1})-\sigma (V_n)\overset{d}{=} \sigma (V_{n+1}-V_n)= \sigma (E_n)$ (where  we used the fact that $\sigma $ has independent and stationary increments).

The transitions functions solve the following equation:
\begin{align}
\mathcal{D}_t \, p_{ij}(t) -\overline{\nu}(t)p _{ij}(0) = \sum _{l\in S} \lambda _i (H_{il}-\delta _{il})p_{lj}(t), \qquad p_{ij}(0)& =\delta_{ij},  \label{equazione frazionaria generalizzata}    
\end{align}
where $\lambda _i$ are the rates of the $E_n$, while  $\nu$ is the L\'evy measure of the subordinator $\sigma$, $\overline{\nu}(t)= \int _t^\infty \nu (dy)$  and
\begin{align*}
\mathcal{D}_t \, p_{ij}(t)=\int _0^\infty  \bigl ( p_{ij}(t)-p_{ij}(t-\tau) \bigr)\nu (d\tau)= \frac{d}{dt}\int _0^t p_{ij}(t-\tau)\overline{\nu}(\tau)d\tau.
\end{align*}
Equation \ref{equazione frazionaria generalizzata} is analogous to \ref{equazioni kolmogorov classiche}, but the time derivative $\frac{d}{dt}$ on the left side is replaced by the operator $\mathcal{D}_t$, which is sometimes called generalized fractional derivative. The reason of this name is that in the case where $\sigma$  is an $\alpha$-stable subordinator, we have $\nu (dx)= \alpha x^{-\alpha-1}dx/\Gamma (1-\alpha)$ and  \ref{equazione frazionaria generalizzata} reduces to 
\begin{align}
\frac{d^\alpha}{dt^\alpha}p_{ij}(t)-\frac{t^{-\alpha}}{\Gamma (1-\alpha)}p _{ij}(0)= \sum _{l\in S} \lambda _i (H_{il}-\delta _{il})p_{lj}(t), \qquad p_{ij}(0)&=\delta_{ij},  \label{equazione frazionaria vera}   
\end{align}
where 
\begin{align*}
\frac{d^\alpha}{dt^\alpha}p_{ij}(t) = \frac{d}{dt}\int _0^t p_{ij}(t-\tau) \frac{\tau^{-\alpha}}{\Gamma (1-\alpha)} d\tau
\end{align*}
is the Riemann--Liouville fractional derivative.
In such a case, the waiting times  have  distribution
\begin{align}
P(J_n>t|   X_n =i)= \mathcal{E} (-\lambda _i t^\alpha ) \label{Mittag Leffler distribution}
\end{align}
where $ \mathcal{E}(x)= \sum _{k=0}^\infty \frac{x^k}{\Gamma (1+\alpha k)}$ is the Mittag--Leffler function; indeed, by a simple conditioning argument, it is easy to check that  the composition  $J_n = \sigma (E_n)$ has Laplace transform
\begin{align}
\mathbb{E}\bigl (e^{-sJ_n}|  X_n=i \bigr )   = \frac{\lambda _i}{\lambda _i+s^{\alpha }} \qquad s\geq 0, \label{trasformata Mittag--Leffler}
\end{align}
which is also the Laplace transform of $-\frac{d}{dt}\mathcal{E} (-\lambda _i t^\alpha )$.
For analytical properties of the Mittag--Leffler function and its role in fractional calculus consult \cite{samko}; see also  \cite{beghin4} and \cite{Garra2} for some applications on relaxation phenomena.

There is an extensive literature concerning the topics recalled in this section: see for example \cite{becker,kolokoltsov3,KoloCTRW,kolokoltsov,merbook,meerscheflimctrw,meerstra,meertri,meertoa}.
See also \cite{scalas} for some applications and \cite{kolokoltsov4} for recent analytic results on fractional differential equations. Moreover, semi-Markov models of motion in heterogeneous media are studied in \cite{Garra,o,ricciuti}, where fractional equations of type \ref{equazione frazionaria vera}   of state dependent order $\alpha =\alpha (x)$ arise. Consult also \cite{orsrictoapota} where the authors study Markov processes time-changed by independent inverses of  additive  subordinators.

\begin{os}
An important case is that of processes making jumps of height 1, i.e. $H _{ij}= 1$ if $j=i+1$ almost surely.  If besides $\lambda _i= \lambda $  for all $i \in \mathcal{S}$, they can be constructed as Poisson processes time-changed by  inverses of subordinators (this is a class of  renewal counting processes including the so-called fractional Poisson process studied e.g. in  \cite{Beghin2,Beghin,Kumar1,mainardi,meerpoisson}). Other models of fractional point processes are studied in \cite{polito 1,polito 2,polito 3}.
Thus, our investigation in the discrete time starts from the discrete analogue of renewal processes, known as renewal chains, which are treated in the next subsection.
\end{os}

\subsection{Renewal chains}

In the following we will make extensive use of generating functions. We recall that    the generating function of a real sequence $\{ a_t \}_{t \in \mathbb{N}_0}$ is defined by the power series
\begin{align*}
\mathcal{G}_a(u)= \sum _{t=0}^\infty u^t a_t
\end{align*}
for all $u \in \mathbb{R}$ such that $|u|\leq R$,  where $R\geq 0$ is the radius of convergence. Since $\mathcal{G}_a$ can be differentiated term by term at all $u$ inside the radius of convergence,  the sequence $\{ a_t \}_{t \in \mathbb{N}_0}$ can be uniquely reconstructed from the generating function by setting
$
a_t= \mathcal{G}_a^{(t)}(0)/t!
$,
where $\mathcal{G}_a^{(t)}(\cdot)$ denotes the $t^{th}$ derivative of $\mathcal{G}_a(\cdot)$. A useful property is that the convolution of two sequences $\{ a_t \}_{t \in \mathbb{N}_0}$ and $\{ b_t \}_{t \in \mathbb{N}_0}$, which is defined as
$ \{a*b\}_{t\in \mathbb{N}_0} = \sum _{k=0}^t a_k b_{t-k}$, 
has generating function
$\mathcal{G}_{a*b}(u)= \mathcal{G}_a(u)\mathcal{G}_b(u)$.

We now recall the notion of renewal chain; for a deeper insight consult  [\cite{barbu}, Chapter 2].
Let $W$ be a positive, integer valued random variable. Let $W_0, W_1, \dots$, be i.i.d.  copies of $W$. Consider the increasing random walk
\begin{align}
T_n=W_0+W_1+\dots W_{n-1} \label{renewal chain Tn} \qquad n\in \mathbb{N} \qquad W_j\in \mathbb{N}    \qquad T_0=0.
\end{align}
Obviously  \ref{renewal chain Tn} has the same form of \ref{subordinatore discreto}, but here,
 intuitively, the $T_n$  should be seen as the successive instants when a specific
event occurs (and we call them renewal times) while the $W_n$ represent the waiting times. 
The process $\{T_n\}_{n\in \mathbb{N}_0}$ is  called  renewal chain or discrete-time renewal process.

Its inverse 
\begin{align}
C(t)=\max \{n\in \mathbb{N}_0: T_n\leq t     \}, \qquad t\in \mathbb{N}_0,  \label{processo di conteggio}
\end{align}
 is a counting process representing the number of renewals up to time $t$.
We now recall a useful result.

\begin{prop}
Let
$\mathcal{G}_{C }(m,u)= \sum _{t=0}^\infty u^t P(C(t)=m),\, |u|<1$,
be the generating function of the sequence $\{P(C(t)=m)\}_{t \in \mathbb{N}_0}$. 
Then
\begin{align}
\mathcal{G}_{C}(m,u)= \frac{1}{1-u}(\mathbb{E}u^W)^m (1-\mathbb{E}u^W). \label{formula utile per i counting processes}
\end{align}
\end{prop}
\begin{proof}
Observe  that $\{ C(t)\geq m\} $ if and only if $\{T_m\leq t\}  $. Then $P(C(t)=m)= P(T_m\leq t, T_{m+1}>t)$. Furthermore, since $\{T_{m+1}\leq t\}$ implies $\{T_m\leq t\}$ we have   $P(C(t)=m)=P(T_m\leq t)-P(T_{m+1}\leq t)$, and computing the generating function of both members, \ref{formula utile per i counting processes} is obtained.
\end{proof}
We now focus on two types of renewal chains together with their related counting processes, which are called Bernoulli and Sibuya  counting processes, respectively. 
In the continuous-time limit, they converge to a Poisson process and to an inverse stable subordinator,  respectively.

\subsubsection{The Bernoulli counting process and the Poisson process} \label{section bernoulli}

 Consider a sequence of Bernoulli trials, i.e. independent trials such that at each time step you can record the occurrence of either 1 event (with probability $p$) or of $0$ events (with probability $q=1-p$). Let $N (t)$ be the number of events up to time $t$, which  follows a binomial distribution $P(N(t)=k)= \binom{t}{k}p^kq^{t-k}, k=0,...,t$. The waiting times  between successive events are independent geometric random variables.
 
\begin{defin}
A Bernoulli counting process, denoted by $\{N (t)\}_{t\in \mathbb{N}_0}$, is a counting process of the type \ref{processo di conteggio} such that the waiting times $M_i$, $i=0,1,\dots,$ have common geometric distribution $P(M_i=k)=pq^{k-1}, k\in \mathbb{N}$.
\end{defin}

By using \ref{formula utile per i counting processes} and the fact that $\mathbb{E}u^{M_i}= \frac{pu}{1-qu}$, the generating function of $\{P(N(t)=m)\}_{t\in \mathbb{N}_0}$ reads
\begin{align}
\mathcal{G}_{N}(m,u)= \frac{(pu)^m}{(1-qu)^{m+1}}.\label{funzione generatrice bernoulli}
\end{align}
Note that $\{N (t)\}_{t\in \mathbb{N}_0}$ is  a Markov process due to the independence of trials and  to the lack of memory property of geometric distribution.

Let now $p_k(t)= P(N(t)=k)$. A simple conditioning argument gives
\begin{align}
p_k(t)= q\, p_k (t-1)+p\, p_{k-1}(t-1)\qquad t\in \mathbb{N}. \label{prima equazione base del processo di Bernoulli}
\end{align}
Re-writing the first member as $p_k(t)= q\, p_k(t)+p\, p_k(t)$ and then dividing by $q$, one obtains a finite difference equation governing the Bernoulli counting process:
\begin{align}
(I-\mathcal{B})p_k(t)= -\lambda p_k(t)+\lambda p_{k-1}(t-1), \qquad t\in \mathbb{N}, \label{seconda equazione base del processo di Bernoulli}
\end{align}
where $\lambda =p/q$, $\mathcal{B}$ is the shift operator acting on the time variable such that $\mathcal{B}p(t)= p(t-1)$ and $I-\mathcal{B}$ is the discrete derivative acting on the time variable.

It is well-known that  the Bernoulli process converges to the Poisson process under a suitable scaling limit. To see this  intuitively (rigorous results on convergence will be given later), let the time steps have size $1/n$ (and thus  the geometric waiting times scale as $M_i\to M_i/n$) and  let the parameter $p$ of the geometric distribution be substituted by $\lambda/n$. In the limit $n\to \infty$ each $M_i$ converges to an exponential  random variable of parameter $\lambda$.

Moreover, one could formally observe that, under the above scaling, equation \ref{seconda equazione base del processo di Bernoulli} reads
\begin{align}
p_k(t)-p_k\bigl (t-\frac{1}{n} \bigr )= -\frac{\lambda}{n} p_k(t)+\frac{\lambda}{n} p_{k-1}\bigl (t-\frac{1}{n}\bigr ), \qquad t\in  \biggl \{\frac{1}{n}, \frac{2}{n},\dots \biggr \}. \label{terza equazione base del processo di Bernoulli}
\end{align}
Dividing by $1/n$ and letting $n\to \infty$, we formally get
\begin{align*}
\frac{d}{dt}p_k(t)= -\lambda p_k(t)+\lambda p_{k-1}(t), \qquad t\in \mathbb{R}^+,
\end{align*}
which is the forward equation governing the Poisson process.

\subsubsection{The Sibuya counting process and the inverse stable subordinator}\label{paragrafo sibuya}

Consider now a sequence of  trials, each having  two possible outcomes, such that the probability of success is not constant in time. If a success has just occurred, we assume that  the probability of success in the $r$-th successive trial is equal to  $\alpha /r$, where $\alpha \in (0,1)$. Unlike the Bernoulli process, which is Markovian, such a process has a memory, since it remembers the time elapsed from the previous success. In place of the geometric distribution of the Bernoulli process,  the time $Z$ in which the first success occurs here follows  the so-called  $Sibuya(\alpha)$ distribution (consult \cite{devroye,pillai})
\begin{align}
P(Z=k)& = (1-\alpha)\bigl (1-\frac{\alpha}{2} \bigr )....\biggl (1-\frac{\alpha}{k-1} \biggr )\frac{\alpha}{k} \notag \\
& = (-1)^{k-1}\binom{\alpha}{k} \qquad k=1,2,\dots \label{distribuzione di Sybuya}
\end{align}
having generating function
\begin{align}
\mathbb{E}u^Z= 1-(1-u)^\alpha \label{funzione generatrice Sibuya}.
\end{align}
On this point, it is interesting to note that  any discrete density  $\{p_k\}_{k\in \mathbb{N}}$ can be expressed as 
\begin{align*}
p_k=(1-\alpha _1)(1-\alpha _2)\dots (1-\alpha _{k-1})\alpha _k \qquad k\in \mathbb{N},
\end{align*}
for  $\alpha _k = p_k/ \sum _{r=k}^\infty p_r$, and thus an interpretation within the sequential trial scheme makes sense. The Sibuya case is obtained by assuming $\alpha _k=\frac{\alpha}{k}$.

Unlike the geometric distribution, which has an exponential decay, the Sibuya distribution has power-law decay, as (see formula 2.5 in \cite{merbook})
\begin{align}
P(Z=k)= (-1)^{k-1}\binom{\alpha}{k} \sim \frac{\alpha}{\Gamma (1-\alpha)}k^{-1-\alpha}, \qquad \textrm{as} \, \, k\to \infty.
\end{align}

Now, we consider a random walk of type \ref{subordinatore discreto} defined by the partial sums of i.i.d. Sibuya random variables
\begin{align}
\sigma _\alpha (n) = \sum _{j=1}^n Z_j \qquad n\in \mathbb{N}, \qquad \sigma _d(0)=0.\label{sub stabile discreto}
\end{align}
We define the process $\{L_\alpha(t)\}_{t\in \mathbb{N}_0}$ as the inverse of $\{\sigma_\alpha (t)\}_{t\in \mathbb{N}_0}$:
\begin{align}
L_\alpha(t)= \max \{ n\in \mathbb{N}_0  : \sigma _\alpha (n) \leq t    \} \qquad t \in \mathbb{N}_0.\label{inverso sub stabile discreto}
\end{align}

 We call $\{L_\alpha (t)\}_{t\in \mathbb{N}_0}$ Sibuya counting process, as it counts the number of successes up to time $t$ in the case of Sibuya trials. By \ref{formula utile per i counting processes} and \ref{funzione generatrice Sibuya}, the sequence $\{P(L_\alpha(t)=m)\}_{t \in \mathbb{N}_0}$ has generating function
\begin{align}
\mathcal{G}_{L_\alpha }(m,u)= (1-u)^{\alpha -1} \bigl [1-(1-u)^\alpha \bigr]^m. \label{funzione generatrice processo sibuya}
\end{align}
Note that \ref{funzione generatrice processo sibuya} can be expanded as
\begin{align*}
\mathcal{G}_{L_\alpha }(m,u)& = \, \sum _{t=0}^\infty u^t \biggl ( \sum _{r=0}^m \binom{m}{r}(-1)^r(-1)^t \binom{\alpha r+\alpha -1	}{t} \biggr )\\
&= \sum _{t=0}^\infty u^t \, \biggl ( \sum _{r=0}^m (-1)^r \binom {m}{r} \binom {t-\alpha r -\alpha	}{t}\biggr ),
\end{align*}
where in the last step we used that $(-1)^w\binom{a-1}{w}= \binom{w-a}{w	}$ for $a>0$. Hence $L_\alpha (t)$ has discrete density
\begin{align}
P(L_\alpha (t)=m)=\sum _{r=0}^m (-1)^r \binom {m}{r} \binom {t-\alpha r -\alpha	}{t},\qquad m=0,1,\dots. \label{distribuzione processo sibuya}
\end{align}

A crucial point is that the Sibuya counting process $\{L_\alpha (t)\}_{t\in \mathbb{N}_0}$ is a discrete-time approximation of the inverse stable subordinator. We remind that a stable subordinator $\{\sigma _*(t)\}_{t\in \mathbb{R}^+}$ is an increasing L\'evy process with Laplace transform $\mathbb{E}e^{-\xi \sigma _* (t)}= e^{-t\xi ^\alpha}$, $\alpha \in (0,1)$, while its inverse hitting time process is defined as
\begin{align}
L_*(t)= \inf \{x: \sigma _* (x)>  t\} = \sup \{x: \sigma _*(x)\leq t     \}. 
\end{align}
The following Proposition, regarding convergence of the one dimensional distribution,  is an anticipation of a more general result that will be treated in Proposition \ref{prop convergenze} where we will state that a suitable scaling of $\{\sigma_\alpha (\lfloor t\rfloor)\}_{t\in \mathbb{R}^+}$ and $\{L_\alpha (\lfloor t\rfloor)\}_{t\in \mathbb{R}^+}$ (where $\lfloor t\rfloor$ denotes the biggest integer less than or equal to $ t$) respectively converge to a stable subordinator and its inverse in the Skorokhod $J_1$ sense and also in the sense of finite dimensional distributions.

\begin{prop}
Let $n \in \mathbb{N}$, $t\in \mathbb{R}^+$. For any $t$, the random variable $n^{-\alpha} L_\alpha (\lfloor nt \rfloor)$ converges in distribution to $L_*(t)$ as $n\to \infty$.  
\end{prop}

\begin{proof}
We remind (see e.g. \cite{dovidio}, formula 1.11) that  $L_*(t)$ has a density:
\begin{align*}
P(L_*(t)\in dx)=\frac{1}{t^\alpha}W_{-\alpha,1-\alpha}\big(-x/t^{\alpha}\big)dx,
\end{align*}
where $W_{\eta,\gamma}(z)=\sum_{r=0}^{\infty}\frac{z^r}{r!\Gamma(\eta r+\gamma)}$  is the Wright function. 
Then, for every $a<b$, with $a,b\in \mathbb{ R}^+$, we have to show that
\begin{equation}\label{EscaledDist}
\lim_{n\rightarrow\infty}P\{a\leq n^{-\alpha}L_\alpha (\lfloor nt \rfloor)\leq b\}=\int_a^b \frac{W_{-\alpha,1-\alpha}\big(-x/t^{\alpha}\big)}{t^{\alpha}}dx.
\end{equation} 
By \ref{distribuzione processo sibuya}, we have 
\begin{eqnarray*}
P(L_{\alpha} (\lfloor nt \rfloor)=s)&=&\sum_{r=0}^{s}\binom{s}{r}(-1)^r\binom{\lfloor{nt}\rfloor-\alpha(r+1)}{\lfloor{nt}\rfloor}\\
&=&\sum_{r=0}^{s}\frac{\Gamma(s+1)}{\Gamma(s-r+1)r!}(-1)^r\frac{\Gamma\big(\lfloor{nt}\rfloor-\alpha(r+1)+1\big)}{\Gamma\big(1-\alpha(r+1)\big)\Gamma\big(\lfloor{nt}\rfloor+1\big)}.
\end{eqnarray*}
Taking $s$ of the form $s=\lfloor{n^{\alpha}h\rfloor}$, with $h\in\bbR^+$, and using Tricomi formula
\begin{align}
 \frac{\Gamma(z+c)}{\Gamma(z+d)}=z^{c-d}\big(1+O(1/z)\big) \qquad c,d\in\bbR, \label{Tricomi formula}
\end{align}
 we write
\begin{eqnarray}\label{Ensim}
P(L_{\alpha} (\lfloor nt \rfloor)= \lfloor{n^{\alpha}h\rfloor})&\sim&
\frac{1}{(nt)^{\alpha}}\sum_{r=0}^{\lfloor{n^{\alpha}h\rfloor}}\Big(\frac{-h}{t^{\alpha}}\Big)^r\frac{1}{r!\Gamma(1-\alpha(r+1))}\Big(1+O\Big(\frac{1}{n^{\alpha}}\Big)\Big)\nonumber\\
&\sim&\frac{1}{(nt)^{\alpha}}\sum_{r=0}^{\infty}\Big(\frac{-h}{t^{\alpha}}\Big)^r\frac{1}{r!\Gamma(1-\alpha(r+1))}\nonumber\\
&\sim&\frac{1}{(nt)^{\alpha}}W_{-\alpha,1-\alpha}\big(-s/(nt)^{\alpha}).
\end{eqnarray}
We use a regular partition of an interval $(a,b]$.
Let $x_i=(i+\lfloor{an^{\alpha}}\rfloor)/n^{\alpha}$, $0\leq i\leq \ell$, and $\ell=\max\{i:x_i\leq b\}$. Note that $x_{i+1}-x_{i}=1/n^{\alpha}$, and since $L_{\alpha} (\lfloor nt \rfloor)$ is a discrete random variable,
$$
P\{a n^{\alpha}\leq L_\alpha (\lfloor nt\rfloor) \leq b n^{\alpha}\}=\sum_{s=\lceil{an^{\alpha}}\rceil}^{\lfloor{bn^{\alpha}}\rfloor}P[ L_\alpha (\lfloor nt\rfloor)=s].
$$
Substituting,  in \ref{Ensim}, each term with its asymptotic value, we obtain
\begin{equation}\label{3.12}
P\{a n^{\alpha}\leq  L_{\alpha} (\lfloor nt \rfloor)    \leq b n^{\alpha}\}\sim \frac{1}{t^{\alpha}}\sum_{s=\lceil{an^{\alpha}}\rceil}^{\lfloor{bn^{\alpha}}\rfloor}\frac{1}{n^{\alpha}}W_{-\alpha,1-\alpha}\left(\frac{-s}{(nt)^{\alpha}}\right)\sim \frac{1}{t^{\alpha}}\sum_{s=\lfloor{an^{\alpha}}\rfloor}^{\lfloor{bn^{\alpha}}\rfloor}\frac{1}{n^{\alpha}}W_{-\alpha,1-\alpha}\left(\frac{-s}{(nt)^{\alpha}}\right).
\end{equation}
The correspondence between $i$ and $x_i$ is one to one and the interval $[x_0,x_{\ell +1}]$ contains the given $[a,b]$. In addition observe that $x_0\leq a<x_1<x_2<\dots<x_{\ell}\leq b < x_{\ell +1}$, then \ref{3.12} may be written as follows:
\begin{equation}\label{3.13}
P\{a n^{\alpha}\leq L_{\alpha} (\lfloor nt \rfloor)=  \leq b n^{\alpha}\}\sim \frac{1}{t^{\alpha}}\sum_{i=0}^{\ell-1}(x_{i+1}-x_i)W_{-\alpha,1-\alpha}\left(\frac{-x_i}{t^{\alpha}}\right),
\end{equation}
which is  the Riemann sum converging to the integral on the right-hand side of \ref{EscaledDist}.
\end{proof}

We now compute the  auto-correlation function of the process $\{L_{\alpha}(t)\}_{t\in \mathbb{N}_0}$.


\begin{prop} \label{covarianza processo sibuya}
The Sibuya counting process $\{L_{\alpha}(t)\}_{t\in \mathbb{N}_0}$ is such that
\begin{align}
\mathbb{E}[L_\alpha (t)]=&\binom{t+\alpha}{t}-1\label{EEt},\\
\mathbb{E}[L_\alpha (t)^2]=& 2\binom{t+2\alpha}{t}-3\binom{t+\alpha}{t}+1 \label{VEt},\\
\mathbb{E} [L_\alpha (t_1)L_\alpha (t_2)]=&\sum_{\ell=1}^{\min(t_1,t_2)}\binom{\ell+\alpha-1}{\ell}\Big[\binom{t_1-\ell+\alpha}{t_1-\ell}+\binom{t_2-\ell+\alpha}{t_2-\ell}-1\Big].  \label{CovEt}
\end{align}
Moreover, as $t_2\to\infty$, 
\begin{align}
Corr[L_\alpha (t_1), L_\alpha (t_2)]=&\frac{Cov[L_\alpha (t_1), L_\alpha (t_2)]}{\sqrt{\bbV[L_\alpha (t_1)]\bbV[L_\alpha (t_2)]}}\sim C\, t_2^{-\alpha},  \label{CorrEt1}
\end{align}
where $C=C(t_1,\alpha)$.
\end{prop}

\begin{proof}
In the following, we will use that
\begin{align}
(-1)^t \binom{-a}{t} =  \binom{a +t-1}{t}, \qquad a>0, \label{first} 
\end{align}
and
\begin{align}
\sum _{\tau =0}^t (-1)^\tau \binom{-a}{\tau}= (-1)^t \binom{-a -1}{t}= \binom{a +t}{t}. \label{second}
\end{align}
Note that \ref{first} can be easily proved by expanding the binomial coefficients, while \ref{second} can be  checked by computing the generating function of both members.
We now need to derive the mean time spent by the Sibuya random walk \ref{sub stabile discreto} at the location $t$:
\begin{align}
\sum _{x=0}^\infty P(\sigma_\alpha (x)=t)=(-1)^t \binom{-\alpha}{t} = \binom{\alpha +t-1}{t} \label{misura potenziale}.
\end{align}
Formula \ref{misura potenziale} can be easily proved by computing (using \ref{funzione generatrice Sibuya}) the generating function of the left-hand side 
\begin{align*}
\sum _{x=0}^\infty \sum _{t=0}^\infty u^t P(\sigma_\alpha (x)=t)= \sum _{x=0}^\infty \mathbb{E}u^{\sigma_\alpha (x)} = \mathbb{E}u^{\sigma_\alpha (0)}+ \sum _{x=1}^\infty (\mathbb{E}u^Z)^x =(1-u)^{-\alpha},
\end{align*}
which is also the generating function of $(-1)^t \binom{-\alpha}{t}$.
 By using \ref{first}, \ref{second} and \ref{misura potenziale} we obtain
\begin{align*}
\mathbb{E} L_\alpha (t)& = \sum _{x=1}^\infty P(L_\alpha (t)\geq x)= \sum _{x=1}^\infty P(\sigma_\alpha (x)\leq t)= \sum _{\tau =0}^t \sum _{x=1}^\infty P(\sigma_\alpha (x)=\tau) \\&= \sum _{\tau =0}^t \biggl ( (-1)^\tau \binom{-\alpha}{\tau} -\delta _0(\tau) \biggr )= \binom{t+\alpha}{t}-1,
\end{align*}
and, by means of \ref{Tricomi formula}, we have
\begin{align}
\mathbb{E} L_\alpha (t)\sim C_3t^{\alpha } \qquad \textrm{as}\,\,\,\, t\to \infty, \label{asintotica valore atteso}
\end{align}
where $C_3=C_3(\alpha)$. For $t_1 \leq t_2$ we have 
\begin{align*}
\mathbb{E} [L_\alpha (t_1)L_\alpha (t_2)]& = \sum _{x=1}^\infty \sum _{y=1}^\infty P(L_\alpha(t_1)\geq x,\, L_\alpha (t_2) \geq y)\\
&= \sum _{x=1}^\infty \sum _{y=1}^\infty P(\sigma_\alpha (x) \leq t_1,\, \sigma_\alpha (y)\leq t_2)\\
& = \sum _{\tau _1=0}^{t_1} \sum _{\tau _2=0}^{t_2}  \sum _{x=1}^\infty \sum _{y=1}^\infty P(\sigma_\alpha (x) =\tau_1,\, \sigma_\alpha (y)= \tau_2)\\ 
&=  \sum _{\tau _1=0}^{t_1} \sum _{\tau _2=\tau _1}^{t_2} \sum _{x=1}^\infty \sum _{y=x}^\infty P(\sigma_\alpha (x)=\tau _1) P(\sigma_\alpha (y-x)= \tau _2-\tau _1)\\
 & \,\,\, +\sum _{\tau _2=0}^{t_1-1} \sum _{\tau _1= \tau _2+1}^{t_1}\sum _{y=1}^\infty \sum _{x=y+1}^\infty P(\sigma_\alpha (y)=\tau _2) P(\sigma_\alpha (x-y)= \tau _1-\tau _2)
 \end{align*}
 where in the last step we used that $\sigma_\alpha$ has independent and stationary increments. By using \ref{first}, \ref{second}, \ref{misura potenziale}, we have
 \begin{align*}
 \mathbb{E} [L_\alpha (t_1)L_\alpha (t_2)]& =
 \sum _{\tau _1=0}^{t_1} \sum _{\tau _2=\tau _1}^{t_2} \binom{\alpha +\tau_2-\tau _1-1}{\tau _2-\tau _1} \biggl ( \binom{\alpha +\tau _1-1}{\tau _1}-\delta _0(\tau _1) \biggr )\\
& \,\,\,\,\,+ \sum _{\tau _2=0}^{t_1-1} \sum _{\tau _1= \tau _2+1}^{t_1} \binom{\alpha +\tau _1-\tau _2-1}{\tau _1-\tau _2}\biggl ( \binom{\alpha + \tau _2-1}{\tau _2}-\delta _0 (\tau _2) \biggr ) \\
&= \sum _{\tau _1=1}^{t_1} \binom{\alpha +t_2-\tau _1}{t_2-\tau _1} \binom{\alpha +\tau _1-1}{\tau _1}+  \sum _{\tau _2=1}^{t_1-1}  \biggr [ \binom{\alpha +t_1-\tau _2}{t_1-\tau _2}-1\biggl ] \binom{\alpha +\tau _2-1}{\tau _2}\\
&=\sum_{\ell=1}^{t_1}\binom{\ell+\alpha-1}{\ell}\Big[\binom{t_1-\ell+\alpha}{t_1-\ell}+\binom{t_2-\ell+\alpha}{t_2-\ell}-1\Big]
\end{align*}
and this proves \ref{CovEt}.
In the end, \ref{VEt} can be obtained by putting $t_1=t_2$ and doing straightforward calculations (to this scope it is useful to recall formula  1.53 in \cite{samko}, which states  that $\sum _{j=0}^k\binom{\alpha}{j}\binom{\beta }{k-j} = \binom{\alpha +\beta }{k}$).
 By assuming $t_1\leq t_2$, we now investigate the asymptotic behaviour of the correlation function for $t_2 \to \infty$ and $t_1$ fixed. First, by using \ref{first} and \ref{second}, we have  that
\begin{align}
& cov \bigl ( L_\alpha (t_1), L_\alpha (t_2) \bigr )=\notag\\ 
&= \sum_{\ell=1}^{t_1}\binom{\ell+\alpha-1}{\ell}\Big[\binom{t_1-\ell+\alpha}{t_1-\ell}+\binom{t_2-\ell+\alpha}{t_2-\ell}-1\Big] -\biggl (\binom{t_1+\alpha}{t_1}-1\biggr ) \biggl (\binom{t_2+\alpha}{t_2}-1\biggr)\notag\\
&=\sum_{\ell=1}^{t_1}(-1)^\ell \binom{-\alpha}{\ell}\Big[\binom{t_1-\ell+\alpha}{t_1-\ell}+\binom{t_2-\ell+\alpha}{t_2-\ell}-1\Big]-\sum _{\ell=1}^{t_1} (-1)^\ell \binom{-\alpha}{\ell} \biggl (\binom{t_2+\alpha}{t_2}-1\biggr)\notag\\
&=\sum_{\ell=1}^{t_1}(-1)^\ell \binom{-\alpha}{\ell}\Big[\binom{t_1-\ell+\alpha}{t_1-\ell}+\binom{t_2-\ell+\alpha}{t_2-\ell}-\binom{t_2+\alpha}{t_2}\Big]\notag\\
&= \sum_{\ell=1}^{t_1}(-1)^\ell \binom{-\alpha}{\ell}\Big[\binom{t_1-\ell+\alpha}{t_1-\ell}+\frac{\Gamma (t_2-\ell+\alpha+1)}{\Gamma (t_2-\ell+1)\Gamma (\alpha+1)}-\frac{\Gamma (t_2+\alpha+1)}{\Gamma (t_2+1)\Gamma (\alpha +1)}\Big].\notag
\end{align}
By \ref{Tricomi formula} we have
\begin{align}
cov \bigl ( L_\alpha (t_1), L_\alpha (t_2) \bigr ) \sim C_1\qquad \textrm{as} \,\,\,\,\,\,t_2\to \infty, \qquad   \textrm{where $\, C_1=C_1(t_1,\alpha)$}.   \label{asintotica covarianza}
\end{align}
 We further have
\begin{align}
Var \bigl (L_\alpha (t_2)\bigr)= 2 \binom{t_2+2\alpha}{t_2}-\binom{t_2+\alpha}{t_2}^2-\binom{t_2+\alpha}{t_2}\sim C_2(\alpha)t_2^{2\alpha}\qquad \textrm{as}\,\,\,  t_2\to \infty, \label{asintotica varianza}
\end{align}
where in the last step we expanded the binomial coefficients and used again \ref{Tricomi formula}. Putting all together, we have
\begin{align*}
corr \bigl (L_\alpha (t_1), L_\alpha (t_2)\bigr)= \frac{cov \bigl (L_\alpha (t_1), L_\alpha (t_2)\bigr) }{\sqrt{Var (L_\alpha(t_1))}\sqrt{Var (L_\alpha (t_2))}} \sim C\, t_2^{-\alpha} \qquad t_2\to \infty.
\end{align*}
where $C=C(t_1,\alpha)$.

\end{proof}

\section{Discrete-time semi-Markov chains}\label{main}

Consider a discrete-time homogeneous Markov chain $\{\mathcal{X}(t)\}_{t \in \mathbb{N}_0}$ on the discrete space $\mathcal{S}$ with transition matrix
\begin{align*}
A_{ij}= P(\mathcal{X}(t+1)=j|\mathcal{X}(t)=i), \qquad i,j\in \mathcal{S}, \qquad \forall t \in \mathbb{N}_0.
\end{align*}
The set of functions
\begin{align*}
P_{ij}(t)= P(\mathcal{X}(t)=j|\mathcal{X}(0)=i) 
\end{align*}
satisfy the backward equation
\begin{align}
P_{ij}(t)= \sum _{k \in \mathcal{S}} A_{ik}P_{kj}(t-1) \label{equazione backward catene}
\end{align}
by virtue of the Chapman-Kolmogorov equality. If the process is in the state $i\in \mathcal{S}$ at time $t$, then at time $t+1$ it remains in the same state with probability $q_i= A_{ii}$, while it makes a jump to a different state with probability $p_{i}=1-A_{ii}$. Thus, if $A_{ii}\in (0,1)$, the waiting time in the state $i$ has geometric distribution of parameter $p_i$, if $A_{ii}=0$ then the waiting time is equal to 1 almost surely (which can be considered a degenerate geometric law), while $A_{ii}=1$ implies that $i$ is an absorbing state and the waiting time is infinity. Under the assumption $A_{ii}< 1$ for each $i\in \mathcal{S}$, equation \ref{equazione backward catene} can be re-written as
\begin{align}
P_{ij}(t)= q_i\, P_{ij}(t-1)+ p_i\, \sum _{l\in S}  H_{il}\, P_{lj}(t-1) \label{equazione Markov 1}
\end{align}
where $H_{ij}$ is the probability of a jump from $i$ to $j$ conditioned to the fact that $i\neq j$. The matrix
\begin{align*}
H_{ij}= \begin{cases}
0 \qquad &j= i \\
\frac{A_{ij}}{1-A_{ii}} \qquad &j \neq i
\end{cases}
\end{align*}
 defines a new Markov chain $\{X_n\}_{n\in \mathbb{N}_0}$ in $\mathcal{S}$:
 \begin{align}
H_{ij}=P(X_{n+1}=j|X_n=i). \label{mm}
\end{align}
By construction, all the waiting times of $\{X_n\}_{n\in \mathbb{N}_0}$ are equal to $1$ almost surely.
Now, the original Markov chain $\{\mathcal{X}(t)\}_{t\in \mathbb{N}_0}$ can also be re-written in the alternative form of a discrete-time jump process having geometric (possibly degenerate) waiting times $M_k$
\begin{align} \label{seconda definizione caso Markov}
\mathcal{X}(t)=X_n \qquad     V_n \leq t<V_{n+1} \qquad \textrm{with}\, \,  V_0=0 ,\, \quad V_n= \sum _{k=0}^{n-1}M_k
\end{align}
with
\begin{align*}
P(M_k=r| X_k=i)= q_i^{r-1}p_i \qquad p_i \in (0,1] \qquad r=1,2,\dots \qquad k=0,1,2\dots
\end{align*}
In some sense, Markovianity of $\{\mathcal{X}(t)\}_{t\in \mathbb{N}_0}$ is a consequence of both markovianity of $\{X_n\}_{t\in \mathbb{N}_0}$ and the lack of memory  of the geometric distribution.

\vspace{0.1cm}

From now on, we will consider discrete-time homogeneous semi-Markov chains, i.e. processes  constructed in the same way of \ref{seconda definizione caso Markov} except for the law of the waiting times, which are now no-longer geometrically distributed:
\begin{align}
\mathcal{Y}(t) =X_n \qquad     T_n \leq t<T_{n+1} \qquad \textrm{where}\quad T_0=0 ,\quad T_n= \sum _{k=0}^{n-1}J_k \label{definizione semi-Markov discreto}
\end{align}
where $H_{ij}$ is the same matrix defined in \ref{mm}, while the waiting times are such that
\begin{align*}
P(J_k=r|X_k=i)= f(r,i)
\end{align*}
for an arbitrary discrete density $f(r,i)$ on $r=1,2,\dots$.       Note that the conditional distribution of the waiting time $J_k$ is time-homogeneous, i.e. depends only on the position of the process and not on the number $k$.   

We are actually interested in two particular subclasses of  \ref{definizione semi-Markov discreto}. Such  subclasses consist  of semi-Markov chains that are constructed as particular modifications of the Markov chain \ref{seconda definizione caso Markov}, in the sense that the new waiting times $J_k$ are given by particular functions of the geometrically distributed waiting times $M_k$.   In the following definition we propose these two models.

\begin{defin} \label{defin semimarkov } Let $\{\mathcal{X}(t)\}_{t\in \mathbb{N}_0}$ be a Markov chain of type  \ref{seconda definizione caso Markov}, having waiting times     $M_k$ with geometric law $P(M_k=r| X_k=i)= p_i q_i^{r-1}, r=1,2,...$, $p_i\in (0,1]$. Let  $\{ \sigma _d (t) \}_{t\in \mathbb{N}_0}$ be  a random walk of type \ref{subordinatore discreto}, independent of $\{\mathcal{X}(t)\}_{t\in \mathbb{N}_0}$,  whose jumps $Z_1, Z_2, \dots$ are i.i.d. copies of a positive integer-valued random variable $Z$.

i) We say that \ref{definizione semi-Markov discreto} is a semi-Markov chain of type A if the waiting times $J_k$ have the compound geometric form
\begin{align}
J_k = \sigma _d (M_k)= \sum _{i=1}^{M_k} Z_i \label{intertempi semi-Markov nel primo caso}
\end{align}
and thus having  generating function
\begin{align}
\mathbb{E} (u^{J_k} | X_k=i)  = \frac{p_i\mathbb{E}u^Z}{1-q_i\mathbb{E}u^Z}.
\end{align}

ii) We say that \ref{definizione semi-Markov discreto} is a semi-Markov chain of type B if the waiting times $J_k$ have the compound shifted geometric form
\begin{align}
J_k= 1+ \sigma _d (M_k-1)=1+\sum _{i=1}^{M_k-1} Z_i \label{bla bla}
\end{align}
and thus having generating function
\begin{align}
\mathbb{E} (u^{J_k} | X_k=i)  = \frac{p_iu}{1-q_i\mathbb{E}u^Z}.\label{bla bla bla bla}
\end{align}
\end{defin}

Observe that if $p_i  \in (0,1)$, then the $J_k$ follow a standard compound geometric law, while in the degenerate case $p_i=1$ we have $M_k=1$  and $J_k=Z_k$ almost surely. Note that, for both semi-Markov chains of type A and B, the waiting times  are delayed with respect to those of the original Markov ones, i.e. $J_k\geq M_k$ almost surely, being the $Z_j$ strictly positive. 

In the special case in which $Z \sim Sibuya (\alpha)$, the random variables $J_k$ follow the discrete Mittag--Leffler distributions of type $A$ and $B$, which will be respectively defined in \ref{trasformata DMLa} and \ref{trasformata DMLb}.

\begin{os} The reason why we are interested in waiting times having compound geometric type distributions  is due to the analogy with the related continuous time  processes treated in the literature. Indeed, as said in Section 2.2, we focused on a subclass of continuous time semi-Markov processes with waiting times following  a compound exponential distribution $\sigma (E_k)$, where   $E_k$ is the exponential waiting time of the original Markov process and $\sigma$ is an independent subordinator (for example, in the case of fractional processes, $\sigma$ is a stable subordinator and $\sigma (E_k)$ follows the Mittag--Leffler distribution). In discrete time, the subordinator $\sigma$ is replaced by the increasing random walk $\sigma _d$ (see \ref{subordinatore discreto}) and the exponential distribution is replaced by the (possibly degenerate) geometric one. With such a choice of compound geometric  waiting times, we will prove that these discrete-time semi-Markov chains  retain important features of the related continuous time processes (and converge to them under  suitable scaling limits). In particular, our type A chains exhibits the property of time-change construction, while type $B$ chains are governed by convolution  equations of generalized fractional type.
\end{os}
 
\subsection{Semi-Markov chains of type A: the time change construction}

\begin{te} \label{teorema sul time change}  Let  $\{ \mathcal{X}(t) \}_{t\in \mathbb{N}_0}$ be a Markov chain of type \ref{seconda definizione caso Markov} having no absorbing states and let $\{\sigma _d(t)\}_{t\in \mathbb{N}_0}$ be an independent random walk of type \ref{subordinatore discreto}, whose inverse is $\{L_d (t)\}_{t\in \mathbb{N}_0}$ defined in \ref{inverso subordinatore discreto}.
 Then the time changed process $\{\mathcal{Y}(t)\}_{t\in \mathbb{N}_0} =\{ \mathcal{X}(L_d(t))\} _{t\in \mathbb{N}_0}$     is a  semi-Markov chain of type A  (according to Definition \ref{defin semimarkov }).
\end{te}

\begin{proof} Since the jumps of $L_d$ are at most of size $1$, both  the jumps of $\mathcal{X}$ and $\mathcal{Y}$ are described by the chain $X_n$. By using definition \ref{seconda definizione caso Markov} we have
\begin{align} 
\mathcal{X}(L_d(t))=X_n \qquad     V_n \leq L_d(t)<V_{n+1} \qquad \textrm{with}\, \,  V_0=0 ,\, \quad V_n= \sum _{k=0}^{n-1}M_k
\end{align}
where each $M_k$ is finite since there are no absorbing  states. Definition \ref{inverso subordinatore discreto} implies that
\begin{align*}
\mathcal{X}(L_d(t))=X_n \qquad     \sigma _d (V_n) \leq t < \sigma _d (V_{n+1}).
\end{align*}
Thus the waiting times are given by
\begin{align*}
J_n= \sigma _d (V_{n+1}) - \sigma _d (V_{n})\overset{d}{=} \sigma _d (V_{n+1}-V_n)=\sigma _d (M_n)= \sum _{i=1}^{M_n} Z_i
\end{align*}
as $\sigma _d $ has independent and stationary increments.

\end{proof}

\subsubsection{Time-changing Markov chains with i.i.d. jumps}  \label{Time-changing Markov chains with i.i.d. jumps}
We now gain more insights on a particular subclass of semi-Markov chains of type $A$.
We  consider Markov chains of type \ref{seconda definizione caso Markov} with values in a discrete state space $\mathcal{S}\subseteq \mathbb{R}$, having independent and stationary increments, i.e. the transition matrix elements $A_{ij}$ only depend on the jump $j-i$.  Such processes can be equivalently written as the  random walk

\begin{align}
X(t)= \sum _{j=1}^t X_j \qquad t\in \mathbb{N} \qquad X(0)=0, \label{random walk}
\end{align}
where $X_1, X_2,...$ are i.i.d. random variables. 
Let $\sigma _d$ be an increasing random walk of type \ref{subordinatore discreto}, independent of  \ref{random walk}, with inverse $L_d$  defined in \ref{inverso subordinatore discreto}.
We are interested in the time-changed process
\begin{align}
Y(t)= X(L_d(t))= \sum _{j=1}^{L_d(t)} X_j \qquad t\in \mathbb{N}, \qquad Y(0)=0. \label{random walk subordinato}
\end{align}
By Theorem \ref{teorema sul time change}, we have that \ref{random walk subordinato} is a semi-Markov chain of type $A$.

The time-change $\ref{random walk subordinato}$ introduces a memory tail effect, which is evident by investigating the  behavior of the auto-correlation function
\begin{align*}
\rho (s,t)=  \frac{cov (Y(t); Y(s))}{\sqrt{Var Y(t)}\sqrt{Var Y(s)}} \qquad s\leq t \quad s\in \mathbb{N} \quad t\in \mathbb{N}.
\end{align*}
A remarkable example of this fact is analyzed in the following proposition, where we prove that in the case where $L_d$ is the Sibuya counting process (defined in \ref{inverso sub stabile discreto}),  then, for fixed $s$ and large $t$, the autocorrelation function of  $Y_\alpha (t)= X(L_\alpha (t))$ exhibits a different decay with respect to the $t^{-\frac{1}{2}}$ decay characterizing the original Markov chain $X$.
 This seems to be useful in many applications, such as the problem of modeling  memory effects in evolving graphs (see \cite{pachon1}, \cite{pachon2}). The following Proposition is the discrete-time counterpart of the analogous result holding in continuous time,  when considering L\'evy processes  time-changed by inverse $\alpha$-stable subordinators (see \cite{correlation} for the computation of the auto-correlation function). This is  consistent with the fact that the Sibuya counting process is just a discrete time approximation of the inverse stable subordinator.

\begin{prop} Let  $\{ X(t) \}_{t\in \mathbb{N}_0}$ be a process of type \ref{random walk}, such that $X_1$ has finite mean and variance, and let $\{ L_d(t) \}_{t\in \mathbb{N}_0}$  be a counting process of type \ref{inverso subordinatore discreto}, independent of $\{ X(t) \}_{t\in \mathbb{N}_0}$. Let $\{ Y(t) \}_{t\in \mathbb{N}_0}$ be the process defined in \ref{random walk subordinato}. Then

a) $\{ Y(t) \}_{t\in \mathbb{N}_0}$   has auto-correlation function
\begin{align*}
\rho (s,t)= \frac{cov (L_d (t); L_d (s))(\mathbb{E}X_1)^2+\mathbb{E}L _d(s)Var X_1}{ \sqrt{Var [L_d (t)](\mathbb{E}X_1)^2+\mathbb{E}L _d(t)Var X_1} \sqrt{Var [L_d (s)](\mathbb{E}X_1)^2+\mathbb{E}L_d (s)Var X_1}} \qquad s \leq t.
\end{align*}

b) Consider the  Sibuya counting process $\{L_\alpha (t)\} _{t\in \mathbb{N}_0}$. Then, for fixed $s$ and large $t$, $\{Y_\alpha (t)\}_{t\in \mathbb{N}_0}= \{X(L_\alpha (t))\}_{t\in \mathbb{N}_0}$ has auto-correlation function
\begin{align*}
\rho_\alpha (s,t)\sim \frac{k_1}{t^\alpha}  \qquad \textrm{if}\qquad \mathbb{E}X_1 \neq 0
\end{align*}
 and
\begin{align*}
\rho_\alpha (s,t)\sim \frac{k_2}{t^{\alpha /2}}  \qquad  \textrm{if}\qquad \mathbb{E}X_1 = 0,
\end{align*}
where $k_1=k_1(s)$ and $k_2=k_2(s)$.
\end{prop}

\begin{proof}
a)
First observe that
\begin{align*}
\mathbb{E}[X(t)X(s)]&= \mathbb{E}[(X(t)-X(s))X(s)]+\mathbb{E}[X(s)^2]\\
&= \mathbb{E}[X(t)-X(s)] \mathbb{E} [X(s)]+ Var [X(s)]+(\mathbb{E}X(s))^2\\
&= (t-s)s (\mathbb{E}X_1)^2+sVar X_1+s^2 (\mathbb{E}X_1)^2\\
&=  ts (\mathbb{E}X_1)^2+sVar X_1
\end{align*}
where we have used independence and stationarity of the increments and the fact that $\mathbb{E}X(t)=t\mathbb{E}X_1$ and $Var X(t)= t Var X_1$. Then, a standard conditioning argument yields
\begin{align*}
\mathbb{E}[X(L_d (t))X(L_d (s))]&= \sum _{w=0}^\infty \sum _{v=0}^\infty \mathbb{E}[X(w)X(v)]P \bigl ( L_d (t)=w, L_d (s)=v \bigr )\\
&= \sum _{w=0}^\infty \sum _{v=0}^\infty  \bigl (  wv (\mathbb{E}X_1)^2+vVar X_1 \bigr )P \bigl ( L_d (t)=w, L_d (s)=v \bigr )\\
&= \mathbb{E} ( L_d (t) L _d (s) ) (\mathbb{E}X_1)^2+ \mathbb{E}L_d (s) Var X_1.
\end{align*}
Taking into account that 
\begin{align*}
\mathbb{E} X(L_d (t))=  \mathbb{E}X_1 \mathbb{E}L_d (t)
\end{align*}
by Wald formula, we have 
\begin{align*}
cov \bigl ( X(L_d (t)); X(L_d (s) ))= cov (L_d (t); L_d (s))(\mathbb{E}X_1)^2+\mathbb{E}L_d (s)Var X_1
\end{align*}
whence, in the special case $s=t$ we have
\begin{align*}
Var [X(L_d (t))]= Var [L_d (t)](\mathbb{E}X_1)^2+\mathbb{E}L_d (t)Var X_1.
\end{align*}
Then the auto-correlation function reads
\begin{align*}
\rho (s,t)&= \frac{cov \bigl ( X(L_d (t)); X(L_d (s) ))}{ \sqrt{Var [X(L_d (t))]} \sqrt{Var [X(L_d (s))]}}\\
&= \frac{cov (L_d (t); L_d (s))(\mathbb{E}X_1)^2+\mathbb{E}L_d (s)Var X_1}{ \sqrt{Var [L_d (t)](\mathbb{E}X_1)^2+\mathbb{E}L_d(t)Var X_1} \sqrt{Var [L_d (s)](\mathbb{E}X_1)^2+\mathbb{E}L_d (s)Var X_1}}.
\end{align*}
b)
Since \ref{asintotica valore atteso}, \ref{asintotica varianza} and \ref{asintotica covarianza} state that for fixed $s$ and large $t$ the following relations hold:
\begin{align*}
cov (L_\alpha (t); L_\alpha (s))\sim C_1 \qquad Var [L_\alpha (t)]\sim C_2 t^{2\alpha} \qquad \mathbb{E}L_\alpha (t) \sim C_3 t^\alpha,
\end{align*}
then by straightforward calculations we obtain the result.
\end{proof}

\subsubsection{Continuous-time limits of discrete-time random walks}

As often mentioned in previous sections, many of the processes studied in this paper are discrete approximation of continuous-time processes: the increasing random walks of type  \ref{subordinatore discreto} and their inverses of type \ref{inverso subordinatore discreto} respectively  converge to subordinators and their inverses, discrete-time Markov chains converge to continuous-time Markov processes, and so forth.
 In this section we prove rigorous results on continuous-time limits. For the notion of Skorokhod $J_1$ and $M_1$ topology, and an exhaustive treaty on path space convergence, consult, for example,  \cite{billingsley}, \cite{skorohodarticolooriginario} and
\cite{whitt}.
 In the following, we denote by $D[0,\infty)$  the space of c\'adl\'ag functions $x:[0,\infty)\to \mathbb{R}$.

As a first result (that we have already anticipated in section \ref{paragrafo sibuya}), we state that the Sibuya random walk \ref{sub stabile discreto} and its inverse \ref{inverso sub stabile discreto} are discrete-time approximations of the stable subordinator and its inverse respectively. The following proposition indeed gives convergence of finite dimensional distributions and also convergence in $J_1$ sense. The proof is not reported since it follows as a special case of the theory given in  \cite{merbook}, \cite{meerscheflimctrw} and all the references therein. The key point of the proof is that is that the Sibuya distribution lies in the domain of attraction of a stable law.
For convenience of the reader, we recall that a random variable $Z$ lies in the domain of attraction of a stable law if, given $Z_1, Z_2, ... , Z_n$ independent  copies of $Z$, it holds that  
\begin{align}
c_n (Z_1+Z_2+...+Z_n)\overset{d}{\to}S, \label{dominio di attrazione}
\end{align}
 where $S$ is stable, for some $c_n \to 0$.

\begin{prop} \label{prop convergenze}
Under the following scaling limit, the Sibuya random walk \ref{sub stabile discreto} and its inverse \ref{inverso sub stabile discreto} respectively converge to a $\alpha -$stable subordinator and to its inverse in the sense of finite dimensional distributions:
\begin{align*}
\{ n^{-\frac{1}{\alpha}}\sigma _{\alpha} (\lfloor nt \rfloor )\}_{t\in \mathbb{R}^+} \overset{fdd}{\longrightarrow} \{\sigma _*(t)\}_{t\in \mathbb{R}^+}, \qquad n\to \infty,
\end{align*}
\begin{align}
\{ n^{-1} L_{\alpha}( \lfloor n^{\frac{1}{\alpha}}t \rfloor) \}_{t\in \mathbb{R}^+} \overset{fdd}{\longrightarrow} \{L_*(t)\}_{t\in \mathbb{R}^+}, \qquad n\to \infty, \label{convergenza inverso sibuya}
\end{align}
where $\lfloor a \rfloor $ denotes the largest integer less than $a$ (or equal to $a$).
 The convergence also holds in   weak sense under the $J_1$ topology on $D[0,\infty)$.
\end{prop}

We now study the continuous-time limit of the same subclass of semi-Markov chains of type A which has been considered in section \ref{Time-changing Markov chains with i.i.d. jumps}, namely those processes obtained by time changing Markov chains having i.i.d. jumps. 

We firstly construct a rescaled version of such processes. Consider a sequence of discrete-time Markov chains with i.i.d. jumps, indexed by the parameter $n$:
\begin{align}
X^{(n)}(t)= \sum _{j=1}^t X^{(n)}_j \qquad t\in \mathbb{N} \qquad X^{(n)}(0)=0. \label{vvv}
\end{align}
Furthermore we  consider a sequence of rescaled random walks with positive jumps of type \ref{subordinatore discreto}, indexed by $n$:
\begin{align}
\sigma _d ^{(n)}(t) = \sum _{j=1}^t Z_j^{(n)}, \qquad t\in \mathbb{N}, \label{Z1}
\end{align}
such that $Z_j^{(n)}$ may now have, in general, real values. Its inverse counting process
\begin{align}
L_d^{(n)}(t)= \max \{ k\in \mathbb{N}_0: \sigma _d^{(n)}(k)\leq t\} \qquad t\in \mathbb{R}^+\label{www}
\end{align}
allows us to define  the time-changed process 
\begin{align*}
Y^{(n)}(t)= X^{(n)}(L_d^{(n)}(t)), \qquad t\in \mathbb{R}^+.
\end{align*}
We also consider
\begin{align*}
X^{(n)}(\lfloor t\rfloor )= \sum _{j=1}^{\lfloor t\rfloor} X^{(n)}_j, \qquad \qquad \sigma _d ^{(n)}(\lfloor t \rfloor) = \sum _{j=1}^{\lfloor t \rfloor} Z_j^{(n)}, \qquad t\in \mathbb{R}^+.
\end{align*}

In the following we denote by $\mathcal{D}(W)$ the set of points of discontinuity of the process $W,$ namely $\mathcal{D}(W)= \{ t>0: W(t^-)\neq W(t)\}$.

\begin{te} \label{teorema convergenze}
If, for $n\to \infty$, the following three conditions hold:

i) $X^{(n)}(\lfloor nt\rfloor )$ converges to a L\'evy Process $A(t)$ in $J_1$ sense,

ii) $\sigma _d ^{(n)}(\lfloor nt \rfloor)$ converges to a subordinator $\sigma (t)$ in $J_1$ sense,

iii) the limit processes $A(t)$ and $\sigma (t)$ are such that $\mathcal{D}(A) \cap \mathcal{D}(\sigma) = \emptyset$ almost surely,
\vspace{0.05cm}

then $Y^{(n)}(t)= X^{(n)}(L_d^{(n)}(t))$ converges in $M_1$ sense to the time changed process $A(L(t))$, 

where $L$ is the inverse  of $\sigma$.
\end{te}

\begin{proof}
We follow the same steps as in [\cite{becker}, Theorem 3.1] and [\cite{meertri}, Theorem 2.1]. We  also make use of the continuous mapping theorem (see [\cite{whitt}, Theorem 3.4.3] and \cite{whitt2}).
By ii), we have that $n^{-1}L_d^{(n)}( t)$ converges to $L(t)$ in $M_1$ topology by a continuous mapping argument (a random function is mapped into its inverse). Then, by another continuous mapping argument, where the couple $\bigl (X^{(n)}(\lfloor nt\rfloor ), n^{-1}L_d^{(n)}( t)\bigr)$ is mapped into the composition, the proof is completed.
\end{proof}

\begin{os}
Theorem \ref{teorema convergenze} applies, for instance, to random walks whose jumps lie in the domain of attraction of a stable law (indeed, following the same steps as in proof of Prop. \ref{prop convergenze}, we actually get the hypotheses i) and ii) on $J_1$ convergence).
But, to be able to understand the importance of Theorem \ref{teorema convergenze}, it is  natural to wonder which other random walks  converge in $J_1$ sense, that is, if there exists some simple criterion to characterize random walks converging  in $J_1$ sense under a suitable scaling limit. One answer is given by Skorokhod in Theorem 2.7 of \cite{skorohod}: a sequence of processes $\xi _n(t)= \sum _{k=1}^{\lfloor nt \rfloor} \xi _k^{n}$, such that the $\xi _k^{n}$ are i.i.d. for each $n$, converges weakly in $J_1$ topology to the process $\xi (t)$ if for each $t$ the random variable  $\xi _n(t)$ converges in distribution to $\xi (t)$.
So, if the addends $X_j^{(n)}$ in \ref{vvv} are such that $X^{(n)}(\lfloor nt\rfloor)\overset{d}{\to}A(t)$  for any $t\in \mathbb{R}^+$, then $X^{(n)}(\lfloor nt \rfloor)$ converges weakly to the L\'evy process $A(t)$  under the $J_1$ topology. 
For suitable sequences $P(X^{(n)}_j=0)$, the limit L\'evy process $A(t)$ will have finite activity, that is it will be   a continuous-time Markov chain with i.i.d. jumps. In the same way, if the positive addends $Z_j^{(n)}$ in \ref{Z1} are such that $\sigma _d^{(n)}(\lfloor nt\rfloor)\overset{d}{\to}\sigma(t)$ for any $t\in \mathbb{R}^+$, then $\sigma _d^{(n)}(\lfloor nt\rfloor)$ converges weakly to a subordinator $\sigma(t)$ in $J_1$ sense.
\end{os}

\begin{os}
Conversely, one could wonder if, given a continuous-time semi-Markov process $A(L(t))$, there exists a discrete-time semi-Markov chain converging to it under a suitable scaling limit. The answer is positive. The first step is to observe that, given the limit processes $A(t)$ and $\sigma (t)$, there exist approximating discrete-time random walks converging to them. This is due to a well-known result on triangular array convergence (see, for example, [\cite{gut}, page 442] and also \cite{gnedenko} for a complete discussion):  since, for each $t$, the random variables $A(t)$ and $\sigma(t)$ are infinitely divisible, there exist i.i.d. random variables $X_k^{(n)}$ and i.i.d. random variables $Z_k^{(n)}$ such that
$X^{(n)}(\lfloor nt\rfloor)=\sum _{k=1}^{\lfloor nt\rfloor} X_k^{(n)}$ converges in distribution to $A(t)$ and $\sigma _d ^{(n)}(\lfloor nt\rfloor)=\sum _{k=1}^{\lfloor nt\rfloor} Z_k^{(n)}$ converges in distribution to $\sigma (t)$ (furthermore Theorem 2.7 in \cite{skorohod}  guarantees also $J_1$ weak convergence). Then, once identified $\sigma _d ^{(n)} $, one  uses its inverse $L_d^{(n)}$ to construct the semi-Markov process $Y^{(n)}(t )= X^{(n)}(L_d^{(n)}(t))$.
\end{os}

\subsection{Semi-Markov chains of type B: generalized fractional finite-difference equations}

We note that, by some algebraic manipulations, equation \ref{equazione Markov 1} governing the Markov chain $\{\mathcal{X}(t)\}_{t\in \mathbb{N}_0}$ can be re-written in the form of a finite difference equation:
\begin{align}
(\mathcal{I-B})P_{ij}(t)&= \sum _{l\in S} \lambda _i (H_{il}\mathcal{B}-\delta _{il})P_{lj}(t), \qquad P_{ij}(0) =\delta_{ij},
\label{equazione Markov 2} 
\end{align}
where $\lambda _i= \frac{p_i}{q_i}$, $\delta _{ij}$ denotes the Kronecker delta, $\mathcal{B} p(t)= p(t-1)$ is the shift operator acting on the time variable, and hence $\mathcal{I-B}$ represents the discrete-time  derivative. 

Equation \ref{equazione Markov 2} is a discrete-time version of equation \ref{equazioni kolmogorov classiche} governing continuous time Markov chains. 
 This fact can be heuristically seen by making use of a suitable scaling limit. Indeed, assume that the time steps have size $1/n$, so that the shift operator acts as $\mathcal{B}_{1/n} p(t)= p(t-1/n)$. Then scale $\lambda _i\to \lambda _i /n $ and divide  both members of \ref{equazione Markov 2} by $1/n$.  The equation reads
\begin{align}
\frac{\mathcal{I-B}_{1/n}}{1/n} \, P_{ij}(t)&= \sum _{l\in S} \lambda _i(H_{il}\, \mathcal{B}_{1/n}-\delta _{il})P_{lj}(t)\qquad    t \in \biggl \{\frac{1}{n}, \frac{2}{n},\dots \biggr \},  \qquad  P_{ij}(0) =\delta_{ij},
\end{align}
and  the continuous time limit $n\to \infty$ gives
\begin{align}
\frac{d}{dt} P_{ij}(t)&= \sum _{l\in S} \lambda _i (H_{il}-\delta _{il})P_{lj}(t), \qquad    P_{ij}(0) =\delta_{ij} \qquad t\in \mathbb{R}^+,
\end{align}
which is the Kolmogorov backward equation \ref{equazioni kolmogorov classiche}.

We now consider a semi-Markov chain of type $B$ (according to Definition. \ref{defin semimarkov })  and we denote by $\gamma (t)$ the (discrete) time spent by the process  in the current position:
\begin{align*}
\gamma (t)= inf \{k \in \mathbb{N} : \mathcal{Y}(t-k)\neq \mathcal{Y}(t)\} \qquad \gamma (0)=1.
\end{align*}
 Starting from a generic renewal time $\tau$ (i.e. such that $\gamma (\tau)=1$), which is a regeneration time for the process,
we derive a system of backward equations for the transition functions
\begin{align*}
p_{ij} (t)& =P(\mathcal{Y}(t+\tau)=j|\mathcal{Y}(\tau)=i, \gamma(\tau)=1)\\
&= P(\mathcal{Y}(t)=j| \mathcal{Y}(0)=i, \gamma  (0)=1) \qquad i,j \in S \quad t \in \mathbb{N}_0,
\end{align*}
where the last equality follows by time homogeneity. Such a system is given by \ref{Equazione frazionariaaa}  of the following theorem.

\begin{te} \label{teorema equazione governante}
Under the initial condition $p_{ij}(0)=\delta _{ij}$, the set of functions $\{ p_{ij}(t), i,j\in S, t\in \mathbb{N}_0\}$ solve the following system of  equations:
\begin{align}
\widetilde{\mathcal{D}}_t\, p_{ij}(t) -P(Z>t)p_{ij}(0)&= \sum _{l\in S} \lambda _i (H_{il}\mathcal{B} -\delta _{il})p_{lj}(t) +\lambda _i p_{ij}(0) \,\delta _{0\,t}, \label{Equazione frazionariaaa}    
\end{align}
where
\begin{align}
\widetilde{\mathcal{D}}_t\, p_{ij}(t)= \sum _{\tau =0}^\infty \bigl ( p_{ij}(t)-p_{ij}(t-\tau) \bigr) P(Z=\tau), \qquad t\in \mathbb{N}_0, \label{derivata frazionaria discreta generalizzata}
\end{align} 
while $\mathcal{B}$ is the shift operator such that $\mathcal{B} p(t)=p(t-1)$ and  $\lambda _i= p_i/q_i$.
\end{te}

\begin{os} It is remarkable to note that \ref{Equazione frazionariaaa} (governing discrete-time semi-Markov chains) can be obtained by \ref{equazione Markov 2} (governing the corresponding Markov ones),  by substituting the discrete-time derivative on the left-hand side with the convolution operator $\widetilde{\mathcal{D}}_t$. Conversely, $\widetilde{\mathcal{D}}_t$  reduces to the discrete derivative $\mathcal{I}-\mathcal{B}$ in the trivial case where $Z=1$ almost surely.

This is analogous to what happens in continuous time, where equation \ref{equazione frazionaria generalizzata} (governing  semi-Markov processes) is obtained from \ref{equazioni kolmogorov classiche} (governing Markov processes) by changing the time derivative with the generalized fractional derivative
\begin{align*}
\mathcal{D}_t\, p_{ij}(t)= \int _0^\infty \bigl ( p_{ij}(t)-p_{ij}(t-\tau) \bigr) \nu (d\tau) \qquad t\in \mathbb{R}^+.
\end{align*} 
Equation \ref{Equazione frazionariaaa} can be interpreted as a discrete-time version of \ref{equazione frazionaria generalizzata}, where the integral in the time variable is replaced  by a series, the L\'evy measure $\nu$ is replaced by the discrete density of $Z$ and the tail of the L\'evy measure $\overline{\nu}$ is replaced by the survival function of $Z$.
\end{os}

\begin{os}
We call the operator \ref{derivata frazionaria discreta generalizzata} generalized fractional discrete derivative. The reason of this name is that in the case where $Z$ follows the Sibuya distribution \ref{distribuzione di Sybuya}, the  operator \ref{derivata frazionaria discreta generalizzata} reduces to the fractional power of the discrete derivative. Indeed, by simple calculations, we have
\begin{align*}
\widetilde{\mathcal{D}}_t\, p_{ij}(t)&= \sum _{\tau =0}^\infty \bigl ( p_{ij}(t)-p_{ij}(t-\tau) \bigr) P(Z=\tau) \\
&= \sum _{\tau =1}^\infty \bigl ( p_{ij}(t)-p_{ij}(t-\tau) \bigr)(-1)^{\tau-1}\binom{\alpha}{\tau}  \\
&= \sum _{\tau =0}^\infty \binom{\alpha}{\tau} (-1)^\tau p_{ij}(t-\tau)\\
&=(\mathcal{I}-\mathcal{B})^\alpha p_{ij}(t),
\end{align*} 
where $\mathcal{B}$ is the shift operator in the time variable, such that $\mathcal{B}p(t)=p(t-1)$.
We observe that such operator also appears in ARFIMA models (see \cite{granger}).
The interested reader can find a pioneering study of the operator \ref{derivata frazionaria discreta generalizzata} in  \cite{mmdn}.
\end{os}

\begin{proof}[Proof of Theorem \ref{teorema equazione governante}]
The discrete-time renewal equation reads
\begin{align*}
p_{ij}(t)=\sum _{\tau=0}^t \sum _{l\in S} H_{il} P(J_0=\tau|X_0=i) p_{lj}(t-\tau) + P(J_0>t|X_0=i)\delta_{ij}.
\end{align*}
By applying the generating function to both members we have
\begin{align}
\tilde{p}_{ij}(u)= \sum _{l\in S} H_{il} \mathbb{E}(u^{J_0}|X_0=i)     \tilde{p}_{lj}(u)+\sum _{t=0}^\infty u^t P(J_0>t|X_0=i)\delta_{ij}. \label{renewal equation trasformata}
\end{align}
Note that for any positive and integer valued random variable $Y$ we have
\begin{align}
\sum _{t=0}^\infty u^t P(Y>t)& = \sum _{t=0}^\infty u^t \sum _{k=t+1}^\infty P(Y=k)=
\sum _{k=1}^\infty \sum _{t=0}^{k-1}u^t P(Y=k) \notag \\ &= \sum _{k=1}^\infty \frac{1-u^k}{1-u} P(Y=k)= \frac{1-\mathbb{E}u^{Y}}{1-u}. \label{funzione di sopravvivenza}
\end{align}
Hence, using \ref{bla bla bla bla} we have
\begin{align*}
\sum _{t=0}^\infty u^t P(J_0>t|X_0=i)=\frac{1-\mathbb{E}(u^{J_0}|X_0=i)}{1-u}=   \frac{1}{1-u}\frac{1-q_i \mathbb{E}u^{Z}-p_iu}{1-q_i\mathbb{E}u^Z}= \frac{p_i+q_i \frac{1-\mathbb{E}u^Z}{1-u}}{1-q_i\mathbb{E}u^Z}.
\end{align*}
Thus \ref{renewal equation trasformata} becomes
\begin{align*}
\tilde{p}_{ij}(u)= \sum _{l\in S} H_{il} \frac{p_iu}{1-q_i \mathbb{E}u^Z}  \,  \tilde{p}_{lj}(u)+  \frac{p_i+q_i \frac{1-\mathbb{E}u^Z}{1-u}}{1-q_i\mathbb{E}u^Z}\delta _{ij} .
\end{align*}
and can be re-written as
\begin{align*}
\tilde{p}_{ij}(u)-q_i \mathbb{E}u^Z \tilde{p}_{ij}(u)= \sum _{l\in S} H_{il}\, p_i\, u\, \tilde{p}_{lj}(u) + \biggl ( p_i+q_i \frac{1-\mathbb{E}u^Z}{1-u} \biggr )\delta _{ij}. 
\end{align*}
Taking the inverse  transform (by using \ref{funzione di sopravvivenza} for the variable $Z$), the previous equation becomes
\begin{align*}
p_{ij}(t)-q _i \sum _{\tau =0}^\infty p_{ij}(t-\tau) P(Z=\tau) = \sum _{l \in S}\, H_{il} \,p_i\, p_{lj}(t-1)+(p_i\delta _{0t}+q_i P(Z>t)\, )\delta _{ij}
\end{align*}
where the  summation is extended to infinity  because $p_{ij}(t)=0$ for $t<0$. Writing $p_{ij}(t)= p_i p_{ij}(t)+q_i p_{ij}(t)$ on the left-hand side, and dividing both sides by $q_i$ (with the position $\lambda _i= p_i /q_i$) we obtain the desired equation.
\end{proof}

\begin{os}
We can retrace the same steps as in the proof of Theorem \ref{teorema equazione governante} in order to find a fractional-type equation governing the Sibuya counting process \ref{inverso sub stabile discreto}. Let $J_0 \sim Sibuya (\alpha)$ and $H_{ij}=1$ if $j=i+1$. We have
\begin{align*}
(I-\mathcal{B})^\alpha p_{(i+1)\, j}(t)  -(-1)^t\binom{\alpha -1}{t}\delta _{ij}   = p_{(i+1)\, j}(t) -p_{ij}(t),
\end{align*}
which is a discrete-time version of equation $\partial ^\alpha _t l(x,t)=-\partial _x l(x,t)+\overline{\nu}(t)\delta _{	ij}$ governing the inverse stable subordinator.
\end{os}

\begin{os}
To be exhaustive, we specify that also semi-Markov chains of type A have  governing equations involving the operator \ref{derivata frazionaria discreta generalizzata}. However, such equations have a cumbersome form, which is certainly not the fractional counterpart of the Markovian equation \ref{equazione Markov 2}. The interested reader can find them by writing the related Markov renewal equation and following the same steps as in the proof of Theorem \ref{teorema equazione governante}.
\end{os}

\subsubsection{Continuous-time limit of the governing equation.}

 Convergence of  \ref{Equazione frazionariaaa} to \ref{equazione frazionaria generalizzata} can be easily obtained in the special case in which $Z$ is regularly varying of order $\alpha \in (0,1)$. For convenience of the reader, we recall below the notion of regular variation (for further details consult \cite{bingham2} and \cite{miko}):
 \begin{defin}
 A non negative random variable $Z$ is said to be regularly varying of order $\alpha$ if its survival function satisfies one of the following equivalent conditions:
 \vspace{0.05cm}
 
 a)  $P(Z>t)=t^{-\alpha}L(t)$, such that $L$ is a slowly varying function (i.e.  $\lim _{n\to \infty} L(nt)/L(n)=1$). 
 \vspace{0.1cm}
 
b) $\lim _{n\to \infty} \frac{P(Z>nt)}{P(Z>n)}=t^{-\alpha}$.
 \vspace{0.1cm}
 \end{defin}
Moreover, observe that the class of regularly varying distributions includes the Sibuya distribution \ref{distribuzione di Sybuya}. If $Z$ is regularly varying, then, under a suitable scaling limit, equation \ref{Equazione frazionariaaa} converges, for $t>0$, to equation \ref{equazione frazionaria vera} governing fractional processes, i.e. Markov processes time changed by an independent inverse stable subordinator. Indeed, by letting the time steps have size $1/n$ and by the scaling
 \begin{align*}
  Z\to \frac{Z}{n} \qquad \lambda _i\to \lambda _i P(Z>n)  \Gamma (1-\alpha)\qquad    \forall i \in \mathcal{S} \qquad \alpha \in (0,1),
\end{align*} 
then,  for  $t\in  \{ \frac{1}{n}, \frac{2}{n}, \dots\}$, equation \ref{Equazione frazionariaaa} reads
 \begin{align}
\sum _{\tau =\frac{1}{n}}^\infty \bigl ( p_{ij}(t)-p_{ij}(t-\tau) \bigr) \frac{1}{\Gamma (1-\alpha)} \frac{P(Z=n\tau)}{P(Z>n)}- \frac{1}{\Gamma (1-\alpha)}\frac{P(Z>nt)}{P(Z>n)}p_{ij}(0)&= \sum _{l\in S} \lambda _i  (H_{il}\mathcal{B}_{\frac{1}{n}}-\delta _{il})p_{lj}(t), \label{equazione slowly varying}
\end{align}
where $\mathcal{B}_{\frac{1}{n}}p(t)=p(t-\frac{1}{n}$).
Taking into account that \footnote{ 
 Since Z is slowly varying, we have
\begin{align*}
\frac{P(Z=nt)}{P(Z>n)} = \frac{P(Z>nt-1)-P(Z>nt)}{P(Z>n)}= \frac{(nt-1)^{-\alpha}L(nt-1)-(nt)^{-\alpha}L(nt)}{n^{-\alpha}L(n)}   = \frac{(nt)^{-\alpha} L(nt)\bigl [(1-\frac{1}{nt})^{-\alpha}\frac{L(nt-1)}{L(nt)}-1 \bigr ]}{n^{-\alpha}L(n)}
\end{align*}
Thus, using the properties of the slowly varying function $L$ and the expansion $(1-\frac{1}{nt})^{-\alpha}\sim 1+\frac{\alpha}{nt}$  for large $n$, the proof is complete. 
}
\begin{align}
\frac{P(Z=nt)}{P(Z>n)}\sim \alpha t^{-\alpha -1}\frac{1}{n}\qquad \textrm{for large $n$}
\end{align}
in the limit $n \to \infty$ equation \ref{equazione slowly varying} reduces to
\begin{align*}
\int _0^\infty \bigl ( p_{ij}(t)-p_{ij}(t-\tau) \bigr) \frac{\alpha \tau ^{-\alpha -1}}{\Gamma (1-\alpha)}d\tau -\frac{t^{-\alpha}}{\Gamma (1-\alpha)}\delta _{ij} =\sum _{l\in S} \lambda _i  (H_{il}-\delta _{il})p_{lj}(t) \qquad t\in \mathbb{R}^+
\end{align*}
which coincides with \ref{equazione frazionaria vera}.

\section{A remarkable special case: the fractional Bernoulli processes}\label{remarkable}

We here analyze  a special case of the theory expounded in the previous section, by constructing two processes which are discrete-time versions of the fractional Poisson process.
We recall that the fractional Poisson process  is a continuous time counting process  whose i.i.d.\  waiting times $J_0, J_1,\dots$, have common law $P(J_k>t)=\mathcal{E}(-\lambda t^\alpha)$, $\alpha \in (0,1)$, where $\mathcal{E}(x)= \sum _{k=0}^\infty \frac{x^k}{\Gamma (1+\alpha k)}$ is the Mittag--Leffler function. As recalled in Section 2.2, it is obtained by the composition of a Poisson process with an  independent inverse stable subordinator.  For some references, consult \cite{Beghin2,Beghin,Kumar1,mainardi,meerpoisson}; see also \cite{Beghin3} and \cite{leonenko2} for  time-inhomogeneous extensions of the model.
The fractional Poisson process is intimately connected to fractional calculus as   its state probabilities solve the  forward equation
\begin{align}
\frac{d^\alpha}{dt^\alpha}p_k(t)-\frac{t^{-\alpha}}{\Gamma (1-\alpha)}p_k(0)=-\lambda p_k(t)+\lambda p_{k-1}(t) \label{equazione fractional Poisson}
\end{align}
where $\frac{d^\alpha}{dt^\alpha}p_k(t) $ denotes the Riemann--Liouville fractional derivative.
We propose two discrete-time approximations of such a process, respectively called fractional Bernoulli of types A and B (in the sense of the classification of semi-Markov chains given in the previous section).
We have been inspired by \cite{pillai}, where the author defines a so-called Discrete Mittag--Leffler distribution. The reason  of this name is that such a distribution converges to the classical continuous Mittag--Leffler distribution \ref{Mittag Leffler distribution} under a suitable scaling limit. We here define two distributions, named discrete Mittag--Leffler distributions of type $A$ and $B$, which are similar to that studied in \cite{pillai}, on which we base the definition of the related Bernoulli processes.

\subsection{Fractional Bernoulli process of type A}

\begin{defin} Let $M$ be a geometric random variable with law $P(M=k)= pq^{k-1}, \,\, k\in \mathbb{N},$ 
and let the $Z_j$ be i.i.d. Sibuya random variables as in \ref{distribuzione di Sybuya}.
\vspace{0.05cm}
 A random variable  $J^A$ is said to follow a $DML_A$ (i.e. discrete Mittag--Leffler of type A) distribution if it can be expressed as a  compound geometric sum 
\begin{align*}
J^A=\sum _{k=1}^M Z_k
\end{align*}
and thus has generating function
\begin{align}
\mathbb{E}u^{J^A}= \frac{1-(1-u)^\alpha}{1+\frac{q}{p}(1-u)^\alpha}, \qquad \alpha \in (0,1). \label{trasformata DMLa}
\end{align}

\end{defin}

\begin{os}
 $DML_A$ is a discrete approximation of the Mittag--Leffler distribution \ref{Mittag Leffler distribution}. Indeed by rescaling 
\begin{align}
Z_k \to \frac{Z_k}{n} \qquad J^A \to \frac{J^A}{n}  \qquad \frac{p}{q}=\lambda \to \frac{\lambda}{n^\alpha}, \label{scaling Mittag--Leffler A}
\end{align}
we obtain the rescaled random variable
\begin{align}
J^{A(n)}= \frac{1}{n}\sum _{j=1}^{M^{(n)}}Z_j  \label{mittag leffler riscalate A}
\end{align}
such  that
\begin{align*}
\lim _{n\to \infty}\mathbb{E}e^{-sJ^{A(n)}}= \frac{\lambda}{\lambda +s^\alpha}   \quad s\in \mathbb{R}^+,
\end{align*}
where $\lambda/(\lambda +s^\alpha)$ is  the Laplace transform of the Mittag--Leffler distribution (see \ref{trasformata Mittag--Leffler}).
\end{os}

We are now ready to define a discrete-time approximation of the fractional Poisson process.

\begin{defin} \label{definizione fractional Bernoulli}
Let $T_n^A= J_0^A+ J_1^A+\dots J_{n-1}^A$ be a renewal chain of type \ref{renewal chain Tn}, with waiting times $J_0^A, J_1^A, \dots,$ having common $DML_A$ distribution. The related counting process 
\begin{align*}
\{N_A(t)\}_{t\in \mathbb{N}_0}=max \{n\in \mathbb{N}_0: T_n^A \leq t\}
\end{align*}
is called fractional Bernoulli counting process of type A.
\end{defin}

In the case $\alpha =1$, we have $Z_k=1$  for each $k$ almost surely, and  $J^A$  defined in \ref{trasformata DMLa}  reduces to a geometric random variable  $J^A  \overset{d}{=} M$ and thus the counting process $N_A$  reduces to the Bernoulli counting process $N$ defined in section \ref{section bernoulli}.

We now give an important time-change relation regarding $\{N_A(t)\}_{t\in \mathbb{N}_0}$.

\begin{prop}
Let  $\{N(t)\}_{t\in \mathbb{N}_0}$ be a Bernoulli counting process and $\{L_\alpha (t)\}_{t\in \mathbb{N}_0}$ be an independent Sibuya counting process defined in \ref{inverso sub stabile discreto}. For each $t\in \mathbb{N}_0$, the following equality holds in distribution
\begin{align*}
 N_A(t)\overset{d}{=}N(L_\alpha(t)).
 \end{align*}

\end{prop}

\begin{proof}
By using \ref{formula utile per i counting processes} with waiting times  \ref{trasformata DMLa}, the generating function of $\{N_A(t)\}_{t\in \mathbb{N}}$ reads
\begin{align*}
\mathcal{G}_{N_A}(m,u)= (1-u)^{\alpha -1} \frac{(p-p(1-u)^\alpha)^m}{(p+q(1-u)^\alpha)^{m+1}}.
\end{align*}
The same form can be obtained by computing the generating function of $N(L_\alpha (t))$ by a simple conditioning argument 
\begin{align*}
\mathcal{G}_{N_A}(m,u)& =\sum _{t=0}^\infty u^t P(N(L_\alpha(t))=m)
= \sum _{t=0}^\infty  \sum _{j=0}^\infty u^t P(N(j)=m) P(L_\alpha(t)=j)\\
&= \sum _{j=0}^\infty  P(N(j)=m)  \mathcal{G}_{L_\alpha}(j,u)
=  (1-u)^{\alpha -1} \frac{(p-p(1-u)^\alpha)^m}{(p+q(1-u)^\alpha)^{m+1}}
\end{align*}
where we  used \ref{funzione generatrice processo sibuya} and \ref{funzione generatrice bernoulli}.
\end{proof}

\vspace{0.05cm}

The above proposition shows that $N_A$ exhibits a time-change construction similar to that of the fractional Poisson process. Indeed, while $N_A$ is given by the composition of a Bernoulli with a Sibuya process,  the fractional Poisson process is given by the composition of a Poisson process with an inverse stable subordinator. The meaning of this construction  is clear if we recall that the Bernoulli and Sibuya processes   converge the Poisson process and  to the inverse stable subordinator respectively.

\vspace{1cm}

\subsection{Fractional Bernoulli process of type B}

\begin{defin} Let $M$ be a geometric random variable with law $P(M=k)= pq^{k-1}, \,\, k\in \mathbb{N},$ 
and let the $Z_j$ be i.i.d. Sibuya random variables as in \ref{distribuzione di Sybuya}.
 A random variable  $J_B$ is said to follow a $DML_B$ (e.g. discrete Mittag--Leffler of type B) distribution if it can be expressed as a compound shifted geometric sum
\begin{align*}
J^B= 1+\sum _{k=1}^{M-1} Z_k
\end{align*}
and thus has generating function
\begin{align}
\mathbb{E}u^{J^B}= \frac{ u}{1 + \frac{q}{p}(1-u)^\alpha}, \qquad \alpha \in (0,1). \label{trasformata DMLb}
\end{align}
\end{defin}

\begin{os}
$DML_B$ is a discrete approximation of the Mittag--Leffler distribution \ref{Mittag Leffler distribution}. Indeed by rescaling 
\begin{align}
Z_k \to \frac{Z_k}{n}  \qquad J^B \to \frac{J^B}{n} \qquad \frac{p}{q}=\lambda \to \frac{\lambda}{n^\alpha}, \label{scaling Mittag--Leffler}
\end{align}
we obtain the rescaled random variable
\begin{align}
  J^{B(n)}=\frac{1}{n}+ \frac{1}{n}\sum _{j=1}^{M^{(n)}-1}Z_j,  \label{mittag leffler riscalate}
\end{align}
such  that
\begin{align*}
 \lim_{n\to \infty}\mathbb{E}e^{-sJ^{B(n)}}= \frac{\lambda}{\lambda +s^\alpha}, \quad s\in \mathbb{R}^+,
\end{align*}
where $\lambda/(\lambda +s^\alpha)$ is  the Laplace transform of the Mittag--Leffler distribution (see \ref{trasformata Mittag--Leffler}).
\end{os}

We now define another discrete-time approximation of the fractional Poisson process.

\begin{defin} 
Let $T_n^B= J_0^B+ J_1^B+\dots J_{n-1}^B$ be a renewal chain of type \ref{renewal chain Tn}, with waiting times $J_0^B, J_1^B, \dots,$ having common $DML_B$ distribution. The related counting process 
\begin{align*}
\{N_B(t)\}_{t\in \mathbb{N}_0}=max \{n\in \mathbb{N}_0: T_n^B\leq t\}
\end{align*}
is called fractional Bernoulli counting process of type B.
\end{defin}

In the case $\alpha =1$, we have $Z_k=1$  for each $k$ almost surely, and   $J^B$ defined in \ref{trasformata DMLb} reduces to a geometric random variable  $ J^B \overset{d}{=} M$ and thus the counting process  $N_B$ reduces to the Bernoulli counting process $N$.

We observe that the process $N_B$ has an interesting connection to fractional calculus, as $p_k(t)= P(N_B(t)=k)$ solves a forward equation which is analogous to \ref{seconda equazione base del processo di Bernoulli} governing $N$, but where the discrete-time derivative $I-\mathcal{B}$ is replaced by its fractional power $(I-\mathcal{B})^\alpha$ :
\begin{align*}
(I-\mathcal{B})^\alpha p(t)= \sum _{k=0}^\infty \binom{\alpha}{k} (-1)^kp(t-k).
\end{align*}

\begin{prop} \label{ricorda}
For $t\in \mathbb{N}_0$, the state probabilities $p_k(t)= P(N_B(t)=k)$ solve the following system 
\begin{align}
(I-\mathcal{B})^\alpha p_k(t)&=-\lambda p_k (t)+\lambda p_{k-1}(t-1), \qquad  k\geq 1, \label{equazione frazionaria con k maggiore di zero}\\
(I-\mathcal{B})^\alpha p_0(t)&-(-1)^t\binom{\alpha -1}{t}= -\lambda p_0(t) +\lambda \delta _{0t}. \label{equazione frazionaria con k uguale a zero} 
\end{align}
under the initial condition $p_k(0)=\delta _{0k}$.
\end{prop}

\begin{proof}
By computing the generating function of both members of \ref{equazione frazionaria con k uguale a zero} one has
\begin{align*}
(1-u)^\alpha \tilde{p}_0(u)-(1-u)^{\alpha -1}=-\lambda \tilde{p}_0(u)+\lambda,
\end{align*}
which gives 
\begin{align*}
\tilde{p}_0(u)= \frac{\lambda +(1-u)^{\alpha -1}}{\lambda +(1-u)^\alpha}.
\end{align*}
By further computing the generating function of both members of \ref{equazione frazionaria con k maggiore di zero} one has
\begin{align*}
(1-u)^\alpha \tilde{p}_k(u)= -\lambda \tilde{p}_k(u)+ \lambda u\, \tilde{p}_{k-1}(u).
\end{align*}
Solving the last equation by iteration, we have
\begin{align*}
\tilde{p}_k(u)& = \frac{\lambda u}{\lambda + (1-u)^\alpha }\tilde{p}_{k-1}(u)
= \frac{(\lambda u)^k}{[\lambda + (1-u)^\alpha]^k} \tilde{p}_0(u)
= \frac{(\lambda u)^k [ \lambda +(1-u)^{\alpha -1}]}{[\lambda +(1-u)^\alpha]^{k+1}},
\end{align*}
which coincides with the generating function obtained by \ref{formula utile per i counting processes} with waiting times  \ref{trasformata DMLb}, and this concludes the proof.
\end{proof}

\begin{os}
Note that for $t\neq 0$ equations \ref{equazione frazionaria con k maggiore di zero} and \ref{equazione frazionaria con k uguale a zero} can be written in compact form as
\begin{align}
(I-\mathcal{B})^\alpha p_k(t) - P(Z>t) p_k(0)=-\lambda p_k (t)+\lambda p_{k-1}(t-1), \qquad t\in  \mathbb{N}, \label{equazione frazionaria compatta}
\end{align}
where $Z$ is the Sibuya random variable defined in \ref{distribuzione di Sybuya}. 

We let  the time steps have size $1/n$ and, following \ref{scaling Mittag--Leffler}, we scale $\lambda \to \lambda/n^\alpha$ and $Z\to Z/n$. Thus \ref{equazione frazionaria compatta} becomes
\begin{align*}
(I-\mathcal{B}_{\frac{1}{n}})^\alpha p_k(t) - P\biggl(\frac{Z}{n}>t \biggr) p_k(0)=- \frac{\lambda}{n^\alpha}\, p_k (t)+\frac{\lambda}{n^\alpha}\, p_{k-1}(t-\frac{1}{n}), \qquad t\in  \biggl \{\frac{1}{n}, \frac{2}{n},\dots \biggr \}, 
\end{align*}
where $\mathcal{B}_{\frac{1}{n}} p_k(t)=p_{k}\bigl (t-\frac{1}{n} \bigr)$,
namely
\begin{align*}
\biggl ( \frac{I-\mathcal{B}_{\frac{1}{n}}}{\frac{1}{n}} \biggr )^\alpha p_k(t) - n^\alpha P(Z>nt)p_k(0)=-\lambda p_k(t)+\lambda p_{k-1}\bigl (t-\frac{1}{n} \bigr ),\qquad t\in  \biggl \{\frac{1}{n}, \frac{2}{n},\dots \biggr \}.
\end{align*}
We let $n \to \infty$ and recall that the operator on the left-hand side 
\begin{align*}
\lim _{n\to \infty}  \biggl ( \frac{I-\mathcal{B}_{\frac{1}{n}}}{\frac{1}{n}} \biggr )^\alpha   
\end{align*}
is known as Gr\"unwald--Letnikov derivative (see [\cite{samko} Chapter 4, sect. 20]), and coincides with the fractional Riemann--Liouville derivative. Moreover, by using \ref{first} and \ref{Tricomi formula}, we have
\begin{align*}
P(Z>nt)=(-1)^{nt}\binom{\alpha -1}{nt}\sim \frac{(nt)^{-\alpha}}{\Gamma (1-\alpha)}, \qquad as \,\,\, n\to \infty.
\end{align*}
 Thus, formally, the limiting  equation coincides with the  forward equation \ref{equazione fractional Poisson} governing the fractional Poisson process.
\end{os}

\begin{os}
We note that \ref{ricorda} is in the forward form. By adapting  \ref{Equazione frazionariaaa}  we can also write the backward equation. Indeed, for $\lambda _i=\lambda$  $\forall i\in \mathcal{S}$, $H_{ij}=1$ if $j=i+1$ and $Z$ following the Sibuya distribution, then the waiting times follow a $DML_B$ distribution (see \ref{trasformata DMLb}) and  \ref{Equazione frazionariaaa}  reduces to
\begin{align*}
(I-\mathcal{B})^\alpha p_{ij}(t) - (-1)^t\binom{\alpha -1}{t} \delta _{ij}= \lambda p_{(i+1),\, j}(t-1)-\lambda p_{ij}(t)+\lambda \delta _{ij}\delta _{0\, t}.
\end{align*}
\end{os}

\subsection{Convergence to the fractional Poisson process.}
We finally prove that, under a suitable limit, both the fractional Bernoulli counting processes $\{N_A(t)\}_{t\in \mathbb{N}_0}$ and  $\{N_B(t)\}_{t\in \mathbb{N}_0}$ defined in this section converge to a fractional Poisson process. The convergence holds in the sense of finite dimensional distributions and also in $J_1$ Skorokhod sense.  

Let $J_0^{A(n)}, J_1^{A(n)},\dots$ be i.i.d. copies of $J^{A(n)}$, defined in \ref{mittag leffler riscalate A}, and let     $T_k^{A(n)}= J_0^{A(n)}+ J_1^{A(n)}+\dots+J_{k-1}^{A(n)}$ be the renewal chain associated with the counting process 
\begin{align*}
N_A^{(n)}(t)= \max \{ y\in \mathbb{N}_0: T^{A(n)}_y\leq t\} \qquad t\in \mathbb{R}^+.
\end{align*}
 Moreover, let $J_0^{B(n)}, J_1^{B(n)},\dots$ be i.i.d.\ copies of $J^{B(n)}$ defined in \ref{mittag leffler riscalate} and let    $T_k^{B(n)}= J_0^{B(n)}+ J_1^{B(n)}+\dots+J_{k-1}^{B(n)}$ be the renewal chain associated with the counting process 
\begin{align*}
 N_B^{(n)}(t)= \max \{ y\in \mathbb{N}_0: T^{B(n)}_y\leq t\} \qquad  t\in \mathbb{R}^+. 
 \end{align*}
 Finally, let  $W_0, W_1,\dots,$ be i.i.d.\ random variables with common Mittag--Leffler law \ref{Mittag Leffler distribution} and let $T_k= W_0+ W_1+\dots+W_{k-1}$ be the renewal process whose counting process is the fractional Poisson process 
 \begin{align*}
 \Pi(t)= \max \{ y\in \mathbb{N}_0: T_y\leq t\},\qquad t\in \mathbb{R}^+.
\end{align*}

\begin{prop} \label{convergenza fdd}
Under the scaling limit \ref{scaling Mittag--Leffler} and \ref{scaling Mittag--Leffler A}, we have that:

a) for $n\to \infty$ ,
\begin{align*}
\{N_A ^{(n)}(t)\}_{t\in \mathbb{R}^+}\overset{fdd}{\longrightarrow} \{\Pi (t)\}_{t\in \mathbb{R}^+}\qquad \{N_B ^{(n)}(t)\}_{t\in \mathbb{R}^+}\overset{fdd}{\longrightarrow} \{\Pi (t)\}_{t\in \mathbb{R}^+}.
\end{align*}
where $\overset{fdd}{\longrightarrow}$ denotes convergence of finite dimensional distributions.

b)  for $n\to \infty$,
\begin{align*}
\{N_A ^{(n)}(t)\}_{t\in \mathbb{R}^+}\overset{J_1}{\longrightarrow} \{\Pi (t)\}_{t\in \mathbb{R}^+}\qquad \{N_B ^{(n)}(t)\}_{t\in \mathbb{R}^+}\overset{J_1}{\longrightarrow} \{\Pi (t)\}_{t\in \mathbb{R}^+} \qquad \textrm{on}\quad D[0,\infty),
\end{align*}
where $\overset{J_1}{\longrightarrow}$ denotes convergence in weak sense under the $J_1$ Skorokhod topology on the space of c\'adl\'ag functions $D[0,\infty)$.

\end{prop}

\begin{proof}

a) The proof proceeds along the same lines for both processes $A$ and $B$. For the sake of brevity, we only show the proof for  case $A$.
Since $J_k^{A(n)}\overset{d}{\longrightarrow}W_k$ for each $k$, where $W_k$ has a Mittag--Leffler distribution \ref{Mittag Leffler distribution}, then, fixed $r\in \mathbb{N}$, the vector $\bigl (J_0^{A(n)}, J_1^{A(n)},\dots, J_{r-1}^{A(n)}\bigr)$ converges in distribution to $\bigl (W_0, W_1,\dots, W_{r-1}\bigr)$. Let $T_k=W_0+W_1+\dots + W_{k-1}$. 
By considering the function $h(x_0,\,x_1,\, \dots , x_{r-1})=(x_0,\, x_0+x_1,\, x_0+x_1+x_2,\, \dots, x_0+x_1+\dots +x_{r-1})$, a continuous mapping argument gives
\begin{align*}
\bigl (T_1^{A(n)},\dots , T_r^{A(n)}\bigr) = h(J_0^{A(n)}, J_1^{A(n)},\dots, J_{r-1}^{A(n)}\bigr) \overset{d}{\to} h\bigl (W_0, W_1,\dots, W_{r-1}\bigr)= \bigl (T_1,\dots , T_r\bigr).
\end{align*}

Hence, by fixing $r$ times $t_1,\dots, t_r$ in $\mathbb{R}^+$ and $r$ numbers $m_1,\dots, m_r$ in $\mathbb{N}$, we have
\begin{align*}
P\bigl ( N_A^{(n)}(t_1)\leq m_1, \dots, N_A^{(n)}(t_r)\leq m_r\bigr)= P\bigl (T^{A(n)}_{m_1}\geq t_1,\dots, T^{A(n)}_{m_r}\geq t_r\bigr)\\
 \overset{n\to \infty}{\longrightarrow}P\bigl (T_{m_1}\geq t_1,\dots, T_{m_r}\geq t_r\bigr)
= P\bigl ( \Pi(t_1)\leq m_1, \dots, \Pi(t_r)\leq m_r\bigr).
\end{align*} 
b) It is sufficient to apply Theorem 3 in \cite{bingham}. Indeed $N_A^{(n)}(t)$ has non decreasing sample paths and, according to the proof of statement a) above, it converges in the sense of finite dimensional distribution to a fractional Poisson process, the latter being continuous in probability.

\end{proof}

\begin{os}\label{abc}
There are many possible discrete-time approximations of the Fractional Poisson process (at least one for each different discrete version of the Mittag--Leffler distribution).
We have only presented two of these possible approximating processes. None of these two is trivially given by sampling the fractional Poisson process at integer times (as well as the Bernoulli process is not given by the  Poisson process sampled at integer times). Indeed, by sampling the fractional Poisson process at integer times, we obtain a process with jumps of size possibly greater than 1. In general, the discrete-time semi-Markov processes treated in this paper are not obtained by sampling the related  continuous-time processes at discrete times.
\end{os}

\vspace{0.7cm}

\textbf{Acknowledgments:} 
F. Polito and C. Ricciuti have been partially supported by the project "Memory in Evolving Graphs" (Compagnia di San Paolo-Università di Torino). F. Polito has been also partially supported by INdAM/GNAMPA.

A. Pachon has been partially supported by the project "Complex Networks and Data Science" (University of South Wales).


\begin{thebibliography}{16}
\providecommand{\natexlab}[1]{#1}
\providecommand{\url}[1]{\texttt{#1}}
\expandafter\ifx\csname urlstyle\endcsname\relax
  \providecommand{\doi}[1]{doi: #1}\else
  \providecommand{\doi}{doi: \begingroup \urlstyle{rm}\Url}\fi



\bibitem[Applebaum(2009)]{apple}
D. Applebaum.
\newblock L\'evy Processes and Stochastic Calculus. Second Edition.
\newblock \emph{Cambridge University Press}, New York, 2009.



\bibitem[Barbu(2008)]{barbu}
V.S. Barbu, N. Limnios.
\newblock{Semi-Markov Chains and Hidden Semi-Markov Models toward Applications}, \emph{Springer}, New York, 2008.


\bibitem[Becker(2004)]{becker} P. Becker-Kern, M. M. Meerschert, H. P. Scheffler.
\newblock{Limit theorems for coupled continuous-time random walks}.
\emph{ Ann. Probab.}, 32, 730-756, 2004.

\bibitem[Beghin(4)]{beghin4}  L. Beghin.    Fractional relaxation equations and Brownian crossing probabilities of a random boundary,	 \emph{Adv. Appl. Probab.},
44(2), 479-505, 2012.


\bibitem {Beghin2} L. Beghin, E. Orsingher. Fractional Poisson processes
and related planar random motions, \emph{Electr. J. Prob.},  14, 1790--1827, 2009.

\bibitem[Beghin (2010) ]{Beghin} L. Beghin, E. Orsingher.
Poisson-type processes governed by fractional and higher-order recursive
differential equations, \emph {Electr. J. Prob.},  22, 684%
--709, 2010.

\bibitem[Beghin3 (2016) ]{Beghin3} L. Beghin, C. Ricciuti.
Time inhomogeneous fractional Poisson processes defined by the multistable subordinator,  \textit{Stoch. Anal. Appl.}, 37(2), pp. 171–188, 2019.



\bibitem [Bertoin(1997)]{bertoins}
J. Bertoin.
\newblock {Subordinators: examples and applications}.
\newblock \emph{Lectures on probability theory and statistics (Saint-Flour, 1997)}, 1-91. \emph{Lectures Notes in Math.}, 1717, Springer, Berlin, 1999.

\bibitem [Bil(1968)]{billingsley}
P. Billingsley.
\newblock {Convergence of Probability Measures.}
 \emph{ John Wiley and Sons}, New York, 1968.
 


 \bibitem [Bingham(1971)]{bingham}
N. H. Bingham.
\newblock {Limit theorems for occupation times of Markov processes}.
 \emph{Wahrscheinlichkeitstheorie and Verw. Gabiete}, 17, 1-22, 1971.
 

\bibitem[Bingham (1987)]{bingham2}
 N. H. Bingham, C. M. Goldie, J. L. Teugels, Regular Variation, Cambridge University Press, 1987.
 
 \bibitem [miko(1971)]{miko}
D. Buraczewski,  E. Damek, T. Mikosch.
\newblock {Stochastic Models with Power-Law Tails}.
 \emph{Springer}, 2016.
 
 
 
 
 \bibitem{cinlarsemi}
E. Cinlar.
\newblock {Markov additive processes and semi-regeneration}.
\newblock \emph{Discussion Paper No. 118}, Northwestern University, 1974.



\bibitem [Cox(197o)]{cox}
D.R. Cox.
\newblock {Renewal Theory}.
\newblock \emph{Methuen \& Co.}, 1970.


\bibitem[devroye (1993)]{devroye} L. Devroye.
A tryptich of discrete distributions related to the stable law.
\emph{Stat. Prob. Lett.}, 18, 349-351, 1993.



\bibitem[dovidio (2015)]{dovidio} M. D'Ovidio.
\newblock{Wright functions governed by fractional directional derivatives and fractional advection diffusion equations}.
\emph{Meth.  Appl. Anal.}, 22, no. 1, 1-36, 2015.


\bibitem{fellersemi}
W. Feller.
\newblock On semi-Markov processes.
\emph{Proceedings of the National Academy of Sciences of the United States of America}, 51(4): 653 -- 659, 1964.


\bibitem[Garra (2015)]{Garra} R. Garra, E. Orsingher, F. Polito. State-dependent fractional point processes. \emph{J. Appl. Probab.},  52, 18-
36, 2015.

\bibitem[Garra2 (2015)]{Garra2} R. Garra, A. Giusti, F. Mainardi, G. Pagnini. Fractional relaxation with time-varying coefficient,  \emph{Frac. Calc. Appl. Anal.}, 17(2), 424-439, 2014.


\bibitem[scalas (2015)]{scalas}
N. Georgiou. I. Z. Kiss, E. Scalas. 
\newblock Solvable non-Markovian dynamic network, 
\emph{Phys. Rev. E} 92, 042801, 2015.

\bibitem[Gnedenko and Kolmogorov(1954)]{gnedenko}
B. V. Gnedenko, N. A. Kolmogorov,
\newblock Limit distributions for sums of independent random variables.
\emph{Addison-Wesley Educational Publishers Inc}, 1954.




\bibitem[Gihman and Skorohod(1975)]{gihman}
I.I. Gihman and A.V. Skorohod.
\newblock The theory of stochastic processes II.
\emph{Springer-Verlag}, 1975.



\bibitem[granger(1980)]{granger}
C.W.J. Granger, R. Joyeux.
An introduction to long memory time series models and fractional differencing,
\emph{J. Time Series Anal.}, (1), pp. 15-29, 1980.





\bibitem[Gut(2013)]{gut}
A. Gut.
\newblock Probability: a graduate course.
\emph{Springer}, New York, 2013.


\bibitem{jacod}
J. Jacod.
\newblock{Syst\`{e}mes r\'egen\'eratifs and processus semi-Markoviens}.
\newblock{\emph{Z. Wahrscheinlichkeitstheorie verw. Gebiete}}, 31: 1-23, 1974.

\bibitem[Kolokoltsov4(2018)]{kolokoltsov4}
M.E. Hernandez-Hernandez, V.N. Kolokoltsov. Probabilistic solutions to nonlinear fractional differential equations of generalized Caputo and Riemann--Liouville type. \emph{Stochastics}, 90(2), pp. 224-255, 2018.


\bibitem[Kolokoltsov3(2009)]{kolokoltsov3}
M.E. Hernandez-Hernandez, V.N. Kolokoltsov, L. Toniazzi. Generalized fractional evolutions equations of Caputo type. \emph{Chaos, Solitons and Fractals}, 102, 184--196, 2017.

\bibitem[Kolokoltsov(2009)]{KoloCTRW}
V.N. Kolokoltsov. Generalized Continuous-Time Random Walks, subordination by hitting times, and fractional dynamics. \emph{Theory Probab.
  Appl.} {53}:, 594--609, 2009.

\bibitem[Kolokoltsov(2011)]{kolokoltsov}
V.N. Kolokoltsov.
\newblock{Markov processes, semigroups and generators}.
\newblock{\emph{de Gruyter Studies in Mathematics, 38. Walter de Gruyter \& Co.}}, Berlin, 2011.


\bibitem[Korolyuk and Swishchuk(1995)]{koro}
V. Korolyuk and A. Swishchuk.
\newblock{Semi-Markov random evolutions}.
\newblock{\emph{Springer-Science+Business Media, B.V.}}, 1995.


\bibitem[Kumar1 (2011)]{Kumar1} A. Kumar, E. Nane, P. Vellaisamy. Time-changed Poisson processes, \emph{Stat. Prob. Lett.}, 81, 1899-1910, 2011.



\bibitem[Kurtz(1971)]{kurtz}
T.G. Kurtz. Comparison of semi-Markov and Markov processes. \emph{ Ann. Math. Stat.}, 42(3): 991 -- 1002, 1971.



\bibitem{correlation} N. N. Leonenko, M.M.\ Meerschaert, R. L. Shilling, A. Sikorskii. \newblock {Correlation structure of time-changed Levy processes} {\it Commun. Appl. Ind. Math.} {6}(2): 1--22, 2014.

\bibitem[leonenko(2017)]{leonenko2}
N. N. Leonenko, E. Scalas, M. Trinh.
\newblock {The fractional non-homogeneous Poisson process}.
\newblock \emph{Stat. Prob. Lett.}, 120, pp. 147-156, 2017.


\bibitem[levy(1956)]{levy}
P. L\'evy. \emph{Processus semi-Markovien,} Proc. Int. Congr. Math. 3, 416-426, 1956.

\bibitem [mainardi(2007)]{mainardi}
F. Mainardi, R. Gorenflo, E. Scalas.
\newblock {A fractional generalization of the Poisson process}.
\newblock \emph{Viet. J. Mathematics}, 32, 53-64, 2007.


\bibitem [Meerschaert and Scheffler(2004)]{meerscheflimctrw}
M.M. Meerschaert and H.P. Scheffler.
\newblock {Limit theorems for continuous-time random walks with infinite mean waiting times}.
\newblock \emph{J. Appl. Probab.}, 41: 623--638, 2004.



\bibitem [Meerschaert and Scheffler(2008)]{meertri}
M.M. Meerschaert and H.P. Scheffler.
\newblock {Triangular array limits for continuous time random walks}.
\newblock \emph{Stoch. Proc. Appl.}, 118(9): 1606-1633, 2008.

\bibitem [Meerschaert et al.(2011)] {meerpoisson}
M.M. Meerschaert, E. Nane and P. Vellaisamy.
\newblock {The fractional Poisson process
and the inverse stable subordinator}.
\newblock \emph{Electr. J.  Probab.}, 16(59): 1600--1620, 2011.


\bibitem[Meerschaert and Sikorskii(2012)]{merbook}
    M.M. Meerschaert and A.  Sikorskii. {\it Stochastic Models for Fractional Calculus}. \emph{De Gruyter},  Berlin, 2012.


\bibitem[Meerschaert and Straka(2014)]{meerstra}
M.M. Meerschaert and P. Straka.
\newblock {Semi-Markov approach to continuous time random walk limit processes}.
\newblock \emph{ Ann. Probab.}, 42(4) : 1699-1723, 2014.



\bibitem [Meerschaert and Toaldo(2015)] {meertoa}
M.M. Meerschaert and B. Toaldo. \newblock{Relaxation patterns and semi-Markov dynamics}.
Stoch. Proc. Appl., 129(8), 2850-2879, 2019.

\bibitem[Michelitsch (2011)]{mmdn}
T. Michelitsch, G. Maugin, S. Derogar, A. Nowakowski. \newblock{Sur une généralisation de l'opérateur fractionnaire}
 (CFM11, Congres Francais de Mecanique 2011, Besancon, France), arXiv:1111.1898v1.


\bibitem [Norris(1998)]{norris}
J.R. Norris.
\newblock {Markov Chains}.
\newblock \emph{Cambridge University Press}, 1998.


\bibitem{polito 1}
E. Orsingher, F. Polito.
\newblock {Fractional pure birth processes}.
\newblock \emph{Bernoulli}, 16(3),  858-881, 2010.


\bibitem{polito 2}
E. Orsingher, F. Polito.
\newblock {On a fractional linear birth-death process}.
\newblock \emph{Bernoulli}, 17(1),  114-137, 2011.


\bibitem{polito 3}
E. Orsingher, F. Polito, L. Sakhno.
\newblock {Fractional non linear, linear and sub-linear death processes}.
\newblock \emph{J. Stat. Phys.}, 141(1),  68-93, 2010.


\bibitem [Orsigher et al.(2016)]{orsrictoapota}
E. Orsingher, C. Ricciuti and B. Toaldo.
\newblock {Time-inhomogeneous jump processes and variable order operators}.
\newblock \emph{Pot. Anal.}, 45(3),  435-461, 2016.

\bibitem [Ors(2018)]{o}
 E. Orsingher, C. Ricciuti and B. Toaldo. \newblock{On semi-Markov processes and their Kolmogorov's integro-differential equations},
\textit{J.  Funct. Anal.}, 275(4),  830-868, 2018.


\bibitem [pachon(1)]{pachon1} A. Pachon, F. Polito and L. Sacerdote. Random Graphs Associated to some Discrete and Continuous Time Preferential Attachment Models. \textit{J. Stat. Phys.}. 162(6), 1608-1638, 2016.


\bibitem [pachon(2)]{pachon2} A. Pachon, L. Sacerdote and S. Yang. Scale-free behaviour of networks with the copresence of preferential and uniform attachment rules. \textit{Physica D: Nonlinear Phenomena}, 371, 1-12, 2018.



\bibitem{pykefinite}
R. Pyke.
\newblock {Markov renewal processes with finitely many states}.
\newblock \emph{Ann. Math. Statist.}, 32:  1243-1259, 1961.


\bibitem{pykeinfinite}
R. Pyke.
\newblock {Markov renewal processes with infinitely many states}.
\newblock \emph{Ann. Math. Statist.}, 35:  1746-1764, 1964.

\bibitem [Pillai (1995)]{pillai}
R.N. Pillai, K. Jayakumar.
\newblock {Discrete Mittag--Leffler distributions}.
\newblock \emph{Stat. Prob. Lett.}, 23, 271-274, 1995.


\bibitem [Raberto et al.(2011)]{raberto}
M. Raberto, F. Rapallo and E. Scalas.
\newblock {Semi-Markov Graph Dynamics}.
\newblock \emph{Plos One}, 6(8): e23370, 2011.


\bibitem[Ricciuti(2017)]{ricciuti}
C. Ricciuti and B. Toaldo. \newblock{Semi-Markov models and motion in heterogeneous media},
\emph{J.  Stat. Phys.}, 169(2): pp. 340-361, 2017.


\bibitem [Samko(1993)] {samko}
S.G. Samko, A.A. Kilbas, O.I. Marichev.
\newblock {Fractional Integrals and Derivatives}.
\newblock \emph{Gordon and Breach Science Publishers}, 1993.


\bibitem [Sato(1999)] {satolevy}
K. Sato.
\newblock {L\'evy processes and infinitely divisible distributions}.
\newblock \emph{Cambridge University Press}, 1999.


\bibitem [Schilling et al.(2010)]{librobern}
R.L. Schilling, R. Song and Z. Vondra\v{c}ek.
\newblock {Bernstein functions: theory and applications}.
\newblock \emph{Walter de Gruyter GmbH \& Company KG}, Vol 37 of De Gruyter Studies in Mathematics Series, 2010.

\bibitem [Skorohod (1956)]{skorohodarticolooriginario}
A. V. Skorohod. \newblock {Limit theorems for stochastic processes,} \emph{Theory  Probab.  Appl.}, 1, 261-290, 1956.


\bibitem [Skorohod (1957)]{skorohod}
A. V. Skorohod. \newblock {Limit theorems for stochastic processes with independent increments,} \emph{Theory  Probab.  Appl.}, 2, 138-170, 1957.

\bibitem [Smith(1955)]{smith}
W. L. Smith. \newblock {Regenerative stochastic processes}, \emph{Proc. R. Soc. London}, Ser. A, 232, 6-31, 1955.


\bibitem [Steutel et al.(20004)]{steutel}
F. Steutel, K. van Harn. \newblock { Infinite divisibility of probability distributions on the real line}, \emph{ Marcel Dekker Inc.}, Basel, 2004, .

\bibitem[strakahenry(2011)]{straka}
P. Straka, B.I. Henry, \newblock{Lagging and leading coupled continuous time random walks, renewal times and their joint limits}, Stoch. Proc. Appl. 121  324-336, 2011.

\bibitem[Toaldo(2014)]{toaldopota}
B. Toaldo.
\newblock {L\'evy mixing related to distributed order calculus, subordinators and slow diffusions.}
\newblock \emph{J. Math Anal. Appl.}, 430(2): 1009 - 1036, 2015.

\bibitem[Whitt(1980)]{whitt2}
W. Whitt.
\newblock {Some useful functions for functional limit theorems,}
\newblock \emph{Mathematics for operations research}, 5(1), 67--85, 1980.

\bibitem[Whitt(2002)]{whitt}
W. Whitt.
\newblock {Stochastic processes limits.}
\newblock \emph{Springer-Verlag}, New York, 2002.



\end{thebibliography}
\end{document}